\documentclass[a4paper,reqno]{amsart}
\usepackage{amssymb}
\usepackage{amsmath}
\usepackage{amsthm}
\usepackage{graphicx}
\usepackage{color}
\usepackage{texdraw}
\usepackage{epsfig}
\usepackage{amsfonts}
\usepackage{calligra}
\usepackage{hyperref}
\usepackage{mathtools}
\usepackage{soul}
\usepackage{tikz}
\usepackage{bm}
\usepackage{caption}

 \allowdisplaybreaks \numberwithin{equation}{section}
\begingroup
\theoremstyle{plain}
\newtheorem{theorem}{Theorem}[section]
\newtheorem{proposition}[theorem]{Proposition}
\newtheorem{lemma}[theorem]{Lemma}

\theoremstyle{definition}
\newtheorem{definition}[theorem]{Definition}
\newtheorem{remark}[theorem]{Remark}

\endgroup

\def \div {\mathop {\rm div}\nolimits}

\def \dist {\mathop {\rm dist}\nolimits}
\def \det {\mathop {\rm det}\nolimits}

\def \de {\mathrm{d}}
\def \e {\epsilon}
\def \re {\mathbb R}

\def \C {\chi}
\def \L {\Lambda}
\def \P {P}

\title[Energy-dissipation balance of a smooth moving crack]{Energy-dissipation balance of a smooth moving crack}
\author[M. Caponi]{Maicol Caponi}
\address[Maicol Caponi]{SISSA, via Bonomea 265, 34136 Trieste, Italy}
\email{mcaponi@sissa.it}
\author[I. Lucardesi]{Ilaria Lucardesi}
\address[Ilaria Lucardesi]{Universit\'e de Lorraine, CNRS, IECL, F-54000 Nancy, France}
\email{ilaria.lucardesi@univ-lorraine.fr}
\author[E. Tasso]{Emanuele Tasso}
\address[Emanuele Tasso]{SISSA, via Bonomea 265, 34136 Trieste, Italy}
\email{etasso@sissa.it}

\begin{document}
\maketitle

\begin{abstract} In this paper we provide necessary and sufficient conditions in order to guarantee the energy-dissipation balance of a Mode III crack, growing on a prescribed smooth path. 
Moreover, we characterize the singularity of the displacement near the crack tip, generalizing the result in \cite{NS} valid for straight fractures.
\end{abstract}

{\small

\medskip
\noindent {\textbf{Keywords:} fracture dynamics, wave equation in time-dependent domains, energy-dissipation balance}

\medskip
\noindent{\textbf{MSC 2010:} 35L05, 35Q74, 74H35, 74R10 
}}
\bigskip

\section{Introduction}

In this paper we compute the kinetic+elastic energy associated to a particular dynamic crack evolution, in which the fracture lips open vertically ({\it anti-plane} case) and the crack set is smooth and preassigned.

We consider as reference configuration a bounded open set $\Omega$ of $\mathbb R^2$ with Lipschitz boundary. We fix a time interval $[0,T]$, a vertical volume force $f$, and we prescribe a boundary deformation on a portion of $\partial \Omega$. We assume that, in response to the external loads, the material breaks along a fixed {$C^{3,1}$} curve $\Gamma\subset \Omega$ with end-points on $\partial \Omega$. In this case, the crack set $\Gamma(t)$ at time $t$ is identified by the crack-tip position on $\Gamma$, described by an arc-length parameter $s(t)$. Here we assume $t\mapsto s(t)$ non decreasing ({\it irreversibility} assumption) and of class {$C^{3,1}([0,T])$}.
Far from the crack set, the material undergoes an elastic deformation: the (vertical) displacement $u$ satisfies a wave equation of the form
\begin{equation}\label{intro-eq1}
\ddot{u}(t)- \div(A \nabla u(t)) = f(t) \quad \hbox{in }\Omega \setminus \Gamma(t)\,,
\end{equation}
where $A$ is a suitable tensor field (satisfying the usual ellipticity conditions); the equation is supplemented by boundary conditions, that we choose Neumann homogeneous on $\Gamma(t)$ ({\it traction free} case), and initial conditions.

The well-posedness of \eqref{intro-eq1} has been widely investigated. We limit ourselves to cite the papers \cite{DM-Lar} and \cite{NS}: in the former, the authors work under the sole assumption of finite measure of the crack set, provide a notion of solution, and show its existence, using a variational time-discretization approach; in the latter, the authors work under stronger regularity assumptions and, following a change of variables approach, prove existence of solutions in a suitable weak sense. Later, in \cite{DM-Luc}, the regular case has been resumed: following the same approach of \cite{NS}, the authors obtain uniqueness of solutions and their continuous dependence on the data. These results have been extended to the vector case in \cite{C}.

In this paper we move the natural step forward in the study: the computation of the kinetic+elastic energy and its relation with the crack growth. This computation has a crucial role, in view of the so-called {\it energy-dissipation balance} which underlies the dynamics (see, e.g., \cite{G,F}): the kinetic + elastic energy released during the elastodynamics and the energy dissipated to create the fracture (the latter proportional to the crack surface increment) balance the work done by the external loads. In formulas, denoting by $\mathcal E(t)$ the energy
\begin{equation}\label{def-e}
\mathcal E(t):= \frac12 \int_{\Omega \setminus \Gamma(t)} \left[ |\dot u(t)|^2 + |\nabla u (t)|^2\right]\, \de t \,,
\end{equation}
and fixing homogeneous Dirichlet-Neumann boundary conditions on $\partial \Omega$, the energy-dissipation balance states that, for every time $t\in [0,T]$,
 \begin{equation}\label{edb}
 \mathcal E(t) - \mathcal E(0) +  \mathcal{H}^1(\Gamma(t) \setminus \Gamma(0)) = \int_0^t \langle f(\tau), \dot{u}(\tau) \rangle_{L^2(\Omega)} \, \de \tau\,.
\end{equation}

The difficulty of computing \eqref{def-e} is twofold: on one hand, the displacement has a singular behavior near the tip; moreover, the domain of integration appearing in \eqref{def-e} is irregular and varies in time. To handle the first issue, a representation result for $u$ is in order: we prove that, for every time $t$, the displacement is of class $H^1$ in a neighborhood of the tip and of class $H^2$ far from it, namely $u$ is of the form
\begin{equation}\label{decomposition}
u(t)=u^R(t) + \zeta(t) k(t) S(\Phi(t))\,,
\end{equation}
where $u^R(t) \in H^2(\Omega \setminus \Gamma(t))$, $\zeta(t)$ is a cut--off function supported in a neighborhood of the moving tip of $\Gamma(t)$, $k(t)\in \mathbb R$, $S\in H^1(\mathbb R^2 \setminus \{x_1\geq 0\})$, and $\Phi(t)$ is a diffeomorphism of $\Omega$ (constructed in a suitable way, according to the properties of $\Gamma$, $A$, and $s$). Once fixed $\zeta$, $S$, and $\Phi$, the function $u^R$ and the constant $k$ are uniquely determined. Actually, the coefficient $k$ only depends on $A$, $\Gamma$, and $s$ (see Theorem \ref{valuek} and Remark \ref{gencase}). In addition, we provide another decomposition for $u$ which is more explicit and better explains the behavior of the singular part (see \S \ref{maicol}). The second issue is technical and we overcome it exploiting Geometric Measure Theory techniques (see Section \ref{sec-err}). The computation leads to the following formula:
\begin{equation}\label{energia}
\mathcal E(t) - \mathcal E(0) + \frac\pi4 \int_0^t  k^2(\tau) {a}(\tau)\dot s(\tau)\, \de \tau  = 
\int_0^t \langle f(\tau), \dot{u}(\tau) \rangle_{L^2(\Omega)} \, \de \tau\,,
\end{equation}
where ${a}$ is a positive function which depends on $A$, $\Gamma$, and $s$, and is equal to 1 when $A$ is the identity matrix; see Theorem~\ref{thmeb} for the proof of \eqref{energia} when $A=I$, and Remark~\ref{genenba} for the general case. By comparing \eqref{edb} and \eqref{energia}, we deduce the following necessary and sufficient condition on the crack evolution (in the class of smooth cracks), in order to guarantee the energy-dissipation balance: during the crack opening, namely when $\dot{s}(t)>0$, the function $k(t)$, often called {\it stress intensity factor}, has to be equal to $2/\sqrt{\pi {a}(t)}$.

We mention that a computation for a horizontal crack $\Gamma(t)=\Omega \cap \{y=0\,,\ x\leq c t \}$ moving with constant velocity $c$ (+ a suitable boundary datum) can be found in \cite[\S 4]{DMLT}.

The representation result stated in \eqref{decomposition} extends that of \cite{NS} for straight cracks (near the tip) and $A$ the identity matrix. Here we adapt their proof to the case of a curved crack and a constant (in time) operator $A$, possibly non homogeneous; moreover, we remove one restrictive assumption on the acceleration $\ddot s$ (see Remark \ref{differenze}). The main steps in the proof of \eqref{decomposition} are the following: performing four changes of variables, we reduce problem \eqref{intro-eq1} to a second order PDE of the form
\begin{equation}\label{intro-eq2}
\ddot{v}(t) - \div (\widetilde{A} (t) \nabla v(t)) + l.o.t. = \widetilde{f}(t) \quad \hbox{in }\widetilde\Omega \setminus \widetilde{\Gamma}_0\,,
\end{equation}
with $\widetilde \Omega$ Lipschitz planar domain and $\widetilde{\Gamma}_0$ a {$C^{3,1}$} curve straight near its tip. The tensor field $\widetilde{A}$ has time-dependent coefficients but at the tip of $\widetilde\Gamma_0$ it is constantly equal to the identity. Finally, the decomposition result for $v$, solution to \eqref{intro-eq2}, obtained via semi-group theory, leads to the one for $u$, solution to the original problem \eqref{intro-eq1}.  

\medskip

The plan of the paper is the following.
In the next section we fix the notations, the standing assumptions on the crack set and on the operator $A$; moreover, we introduce the changes of variables which transform \eqref{intro-eq1} into \eqref{intro-eq2}.
Then, in Section \ref{sec-dec} we adapt the proof of the decomposition result \cite[Theorem 4.8]{NS} to our more general case, underlying the main differences. Finally, in Section \ref{sec-err}, we prove the energy balance \eqref{energia}.

\section{Preliminaries}\label{sec-prel}

\subsection{Notation}
We adopt standard notations for Lebesgue and Sobolev spaces on bounded open sets of $\re^2$. The boundary values of a Sobolev function are always intended in the sense of traces, and the one dimensional Hausdorff measure is denoted by $\mathcal H^{1}$. Given an open set $\Omega$ with Lipschitz boundary, we denote by $n$ the outer unit normal vector to $\partial\Omega$, defined a.e.\ on the boundary. Moreover, given a non negative summable function $w$ in $\Omega$, we denote by $L^p(\Omega,w\de x)$ the weighted $L^p$-space on $\Omega$ with respect to the measure $w\, \de x$.

Given a normed vector space $X$ and its topological dual $X^*$, the norm in $X$ is denoted by $\|\cdot\|_X$ and the duality product between $X^*$ and $X$ is denoted by $\langle \cdot, \cdot \rangle_{X}$. We adopt the same notations also for vector valued functions in $X$.
When no ambiguity may arise, we write $\|\cdot \|_\infty$ for the $L^\infty$-norm of scalar and vector functions, computed in their domain of definition.
Given an interval $I\subset \re$ and a Banach space $X$, $L^p(I; X)$ is the space of $L^p$ functions from $I$ to $X$. Given $u\in L^p(I;X)$, we denote by $\dot{u}\in \mathcal D'(I;X)$ its distributional derivative. 

We write $SO(2)^+$ to represent the space of $2 \times 2$ orthonormal matrices whose determinant is equal to $1$.

\subsection{Standing assumptions}
We consider a bounded open set $\Omega \subset \mathbb R^2$ with Lipschitz boundary $\partial \Omega$, we take a Borel subset $\partial_D\Omega$ of $\partial \Omega$ (possibly empty), and we denote by $\partial_N \Omega$ its complement. We fix a {$C^{3,1}$} curve $\gamma: [0,\ell] \to \overline\Omega$ parametrized by arc-length, with end-points on $\partial \Omega$; namely, denoting by $\Gamma$ the support of $\gamma$, we assume $\Gamma\cap \partial \Omega = \gamma(0) \cup \gamma(\ell)$. 
Let $T>0$ and $s:[0,T]\to (0,\ell)$ be a non decreasing function of class {$C^{3,1}$}. We set
\[
\Gamma(t):=\{\gamma(\sigma) \ :\ 0 \leq \sigma \leq s(t) \}\,.
\]
\begin{figure}[h]                                             
\begin{center}                                                
{\includegraphics[height=3.5truecm] {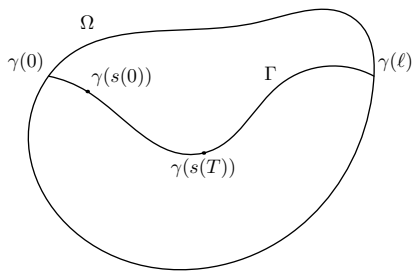}}                                                
\end{center}                                                  
\caption{{\it The endpoints of $\Gamma$ are $\gamma(0)$ and $\gamma(\ell)$ and belong to $\partial \Omega$. We study the evolution of the fracture along $\Gamma$ from $\gamma(s(0))$ to $\gamma(s(T))$.}}\label{fig-taglio} 
\end{figure}                                                  

Given a tensor field $A:\overline{\Omega} \to \mathbb R^{2\times 2}_{sym}$ with {smooth ($C^{2,1}$ would be enough)} coefficients satisfying the ellipticity condition
\begin{equation}\label{ell-A}
(A(x) \xi) \cdot  \xi \geq c_0|\xi|^2 \quad \forall \, x\in \overline \Omega\,, \ \xi \in \mathbb R^2\,,
\end{equation}
a function $f\in C^0([0,T];H^1(\Omega))\cap {\mathrm{Lip}([0,T];L^2(\Omega))}$, and suitable initial data $u^0$ and $u^1$ (for the precise regularity, see Theorem \ref{thm-decomp}), 
we consider the differential equation
\begin{equation}\label{eq-1}
\ddot{u}(t) - \div (A \nabla u(t)) = f(t) \quad \hbox{in }\Omega \setminus \Gamma(t)\,,
\end{equation}
with initial conditions 
$$
u(0)=u^0\,,\ \dot{u}(0)= u^1\,,
$$
and boundary conditions
\begin{equation}\label{bc-u}
u(t)=0 \quad \hbox{on } \partial_D\Omega \,,\quad (A \nabla u(t))\cdot n = 0 \quad \hbox{on }\partial_N\Omega \cup \Gamma(t)\,,
\end{equation}
where $n$ denotes the unit normal vector.
The equation \eqref{eq-1} has to be intended in the weak sense, namely valid for every $t\in [0,T]$ in duality with an arbitrary test function in $H^1(\Omega \setminus \Gamma(t))$ with zero trace on $\partial_D\Omega$ (see also \cite[Definition 2.4]{DM-Luc}).

Furthermore, we assume that the velocity of $s$ is bounded by the constant $c_0$ as follows:
\begin{equation}\label{ipotesi-s}
|\dot{s}(t)|^2 \leq c_0- \delta\quad \forall t\in [0,T]\,,
\end{equation}
for some constant $0< \delta\leq c_0 $. 

The importance of this bound is twofold: on one hand, the relation with the ellipticity constant $c_0$ of $A$ will guarantee the resolvability of the problem (see also \eqref{claim1} in Lemma \ref{lem-A4}); on the other hand, the estimate will allow us to work locally in time, and then, repeating the procedure a finite number of times, to obtain a global result in $[0,T]$.

\subsection{The change of variables approach}\label{ssec-cov}
We fix $t_0,t_1\in[0,T]$ such that $0<t_1-t_0<\rho$, with $\rho$ sufficiently small. A comment on the value of $\rho$ is postponed to Remark \ref{remrho}. In the following, we perform 4 changes of variables: first we act on the operator $A$, transforming it into the identity on the fracture set; then we straighten the crack in a neighborhood of $\gamma(s(t_0))$; then we recall the time-dependent change of variables introduced in \cite{DM-Luc}, that brings $\Gamma(t)$ into $\Gamma(t_0)$ for every $t\in[t_0,t_1]$; finally, we perform a last change of variables in a neighborhood of the (fixed) crack tip, in order to make the principal part of the transformed equation equal to the minus Laplacian. For the sake of clarity, at each step, we use the superscript $i$, $i=1,\ldots, 4$, to denote the new objects: the domain $\Omega^{(i)}$, the fracture set $\Gamma^{(i)}$, and the time-dependent crack $\Gamma^{(i)}(t)$. We will also introduce the tensor fields $A^{(i)}$, which characterize the leading part (with respect to the spatial variables)  $-\div (A^{(i)} \nabla v)$ of the PDE \eqref{eq-1} transformed.

\medskip

\noindent {\it Step 1.}  
Thanks to the standing assumptions on $A$, we may find a tensor field $Q$ of class $C^{2,1}(\overline{\Omega};\mathbb R^{2\times 2})$ such that, for every $x\in \Omega$,
\begin{equation}\label{QAQ}
Q(x) A(x) Q^T(x) = I\,,
\end{equation}
being $I$ the identity matrix. In particular we can choose $Q(x)$ to be equal to the square root matrix of $A^{-1}(x)$, namely $Q(x)=Q^T(x)$ and $Q^2(x)=A^{-1}(x)$. It is easy to prove the existence of a {smooth} diffeomorphism $\C$ {(again, $C^{3,1}$ would be enough)} of $\overline \Omega$ which is the identity in a neighborhood of $\partial \Omega$ and satisfies, at least in a neighborhood $V$ of $\gamma(s(t_0))$, $D\C (x)=Q(x)$ on $\Gamma$. Notice that the constraint $D\C = Q$ cannot be satisfied in the whole domain, since the lines of $Q$ in general are not curl free. 
We set
\begin{align*}
&\Omega^{(1)}:=\Omega\,,\ 
\Gamma^{(1)}:=\C (\Gamma)\,, \  \Gamma^{(1)}(t):=\C (\Gamma(t))\,,
\\ 
& A^{(1)}(x):= [D\C \, A \, D\C^T] (\C^{-1}(x))\,.
\end{align*}
Clearly, the tensor $A^{(1)}$ satisfies an ellipticity condition of type \eqref{ell-A} for a suitable constant $C_1>0$ and it equals the identity matrix on $\Gamma^{(1)}$. 
Moreover, we may easily write an arc-length parametrization $\gamma^{(1)}$ of $\Gamma^{(1)}$ exploiting that of $\Gamma$, by setting 
$$
\gamma^{(1)}:=\C \circ \gamma \circ \beta\,,\quad \hbox{with}\quad 
\beta^{-1}(\sigma) := \int_0^{\sigma} \| (\C \circ \gamma ) \prime\|\, \de\tau\,.
$$
Accordingly, the time-dependent fracture $\Gamma^{(1)}(t)$ is parametrized by 
$$
\Gamma^{(1)}(t) = \gamma^{(1)}(s^{(1)}(t))\,,\quad \hbox{with}\quad s^{(1)}:= \beta^{-1}\circ s\,.
$$
Note that the function $s^{(1)}$ is of class {$C^{3,1}$} and, thanks to \eqref{QAQ} and \eqref{ipotesi-s}, satisfies the following bound:
\begin{equation}\label{bound-s1}
|\dot{s}^{(1)}(t)| ^2= \left|\frac{\de \beta^{-1}}{\de s} (s(t)) \right|^2|\dot{s}(t)|^2
\leq \max_{\|\xi\|=1\,,\,x\in \Gamma\cap V}\| D\C (x) \xi\|^2 |\dot{s}(t)|^2\leq (1-c_1^2)\,,
\end{equation}
where, for brevity, we have set $c_1^2:=\delta/c_0$.
Moreover, for the sake of clarity, we also fix a notation for the maximal acceleration: we set $c_2$ as
\begin{equation}\label{def-c2}
c_2:=\max_{t\in[t_0,t_1]} |\ddot{s}^{(1)}(t)|\,.
\end{equation}
A direct computation proves that $c_2$ is bounded and depends on $c_0$, $\delta$, $\ddot{s}$, $\gamma^{\prime \prime}$, and $D^2\C$.

\medskip

\noindent {\it Step 2.} Now we provide a change of variables $\L$ of class $C^{2,1}$ 
which straightens the crack in a neighborhood of $\gamma^{(1)}(s^{(1)}(t_0))$. First, up to further compose $\Lambda$ with a rigid motion, we may assume that the crack-tip of $\Gamma^{(1)}(t_0)$ is at the origin, and the tangent vector to $\Gamma^{(1)}$ at the origin is horizontal, namely
$$
\gamma^{(1)}(s^{(1)}(t_0)) = 0\,,\quad (\gamma^{(1)})^\prime(s^{(1)}(t_0)) = e_1=(1,0)\,.
$$
For brevity, we set $\sigma_0:=s^{(1)}(t_0)$. We begin by transforming a tubular neighborhood $U$ of the fracture near 0 into a square: exploiting the representation
$$
U = \{ \gamma^{(1)}(\sigma_0+\sigma) + \tau n^{(1)}(\sigma_0+ \sigma)\ :\  \sigma \in (- \varepsilon, \varepsilon)\,,    \tau \in (-\varepsilon, \varepsilon)   \} 
$$
with $n^{(1)} := (\gamma^{(1)})'^\perp$, we define $\L\colon U\to  (-\varepsilon,\varepsilon)^2$ as the inverse of the function $(\sigma,\tau)\mapsto \gamma^{(1)}(\sigma+\sigma_0) + \tau n^{(1)}(\sigma+\sigma_0)$. The global diffeomorphism is obtained by extending $\L$ to the whole $\Omega$. Accordingly, we set 
\begin{align*}
& \Omega^{(2)}:=\Lambda(\Omega^{(1)})\,,\  \Gamma^{(2)}:=\L (\Gamma^{(1)})\,, \ \Gamma^{(2)}(t):=\L (\Gamma^{(1)}(t))\,,
\\
& A^{(2)}(x):= [D\L \, A^{(1)}\, D\L^T] (\L^{-1}(x))\,.
\end{align*}
The tensor field $A^{(2)}$ still satisfies an ellipticity condition of type \eqref{ell-A}, for a suitable constant.

For $x\in\Gamma^{(2)}$ in a neighborhood of the origin, setting  $y:= \L^{-1}(x)\in \Gamma^{(1)}$, we have
$$
 A^{(2)}(x) = D\L (y) \, A^{(1)}(y) \, D\L^T(y) =  D\L (y)  \, D\L^T(y) = [ (D (\L^{-1}))^T(x)\, D(\L ^{-1})(x)  ]^{-1} = I\,.
$$
The last equality follows from
\begin{equation}\label{Lambda}
\frac{\partial (\L^{-1})}{\partial \sigma} (\sigma,\tau) =   (\gamma^{(1)})' (\sigma_0+\sigma)+ \tau (n^{(1)})^\prime (\sigma_0+\sigma)\,,\quad \frac{\partial (\L^{-1})}{\partial \tau}  (\sigma,\tau) = n^{(1)}(\sigma_0+\sigma)\,,
\end{equation}
and the fact that here we consider $x$ of the form $x=(\sigma,0)$.
In particular, we may be more precise on the ellipticity constant of $A^{(2)}$ restricted to a neighborhood of  the origin: for every $0<\e < 1$, there exists $r>0$ such that
\begin{equation}\label{ell-A2}
(A^{(2)}(x) \xi) \cdot \xi \geq (1-\e ) |\xi|^2 \quad \forall \xi \in \mathbb R^2\,,\ \forall |x|\leq r\,.
\end{equation}

Finally, we underline that if $\rho\coloneqq t_1-t_0$ is small enough (see also Remark \ref{remrho}), the whole set $\Gamma^{(1)}(t_1)\setminus \Gamma^{(1)}(t_0)$ is contained in $U$, so that the time dependent fracture $\Gamma^{(2)}(t)$ satisfies
$$
\Gamma^{(2)}(t)= \Gamma^{(2)}(t_0) \cup \{ (\sigma,0)\ :\ 0 \leq \sigma \leq s^{(1)}(t)-s^{(1)}(t_0)\}\,,
$$
for every $t\in [t_0,t_1]$.

\medskip

\noindent {\it Step 3.} Here we introduce a family of 1-parameter $C^2$ diffeomorphisms $\Psi(t,\cdot)$, $t\in [t_0,t_1]$, which transform every $\Omega^{(2)}\setminus \Gamma^{(2)}(t)$ into $\Omega^{(2)}\setminus \Gamma^{(2)}(t_0)$. All in all, we map the non cylindrical domain $\{(x,t)\ :\ x \in \Omega^{(2)}\setminus \Gamma^{(2)}(t)\,,\ t\in [t_0,t_1]\}$ into the cylindrical one $(\Omega^{(2)}\setminus \Gamma^{(2)}(t_0)) \times {[t_0,t_1]}$. This construction can be found in \cite{NS} and \cite[Example 1.14]{DM-Luc}, thus we limit ourselves to recall the main properties: the diffeomorphism $\Psi:[t_0,t_1]\times \overline{\Omega}\to \overline{\Omega}$ satisfies
\begin{align*}
\Psi(t_0)=id\,,\quad \Psi(t)_{\lfloor_{\partial \Omega}}= id\,,\quad \Psi(t)(\Gamma^{(2)}(t)) = \Gamma^{(2)}(t_0)\,,
\end{align*}
being $id$ the identity map. The corresponding tensor field is
\[
A^{(3)}(t,x):= [ D\Psi \, A^{(2)} D\Psi^T -  \dot{\Psi} \otimes  \dot{\Psi}](\Psi^{-1}(t,x))\,.
\]
Note that $A^{(2)}$ does not depend on time, while $A^{(3)}$ does. 

In a neighborhood of the origin, 
\begin{equation}\label{intorno}
\Psi(t,x) = x - (s^{(1)}(t)-s^{(1)}(t_0)) e_1\quad \hbox{and}\quad  \Psi^{-1}(t,x) = x + (s^{(1)}(t)-s^{(1)}(t_0)) e_1\,,
\end{equation}
so that $D\Psi = I$,  $\dot{\Psi}= - \dot{s}^{(1)}e_1$, and for $x=(x_1,0)$ with $x_1$ small enough in 
modulus,
$$
A^{(3)}(t, x) = \left(  \begin{array}{ccc}  1- |\dot{s}^{(1)}(t)|^2\ & 0 \\ 0 \ &  1\end{array}  \right)\,.
$$

\medskip

\noindent {\it Step 4.} In this last step we apply a change of variables $\P$ near the origin (namely the crack-tip of $\Gamma^{(2)}(t_0)$), in order to make the tensor field $A^{(4)}$, constructed as in the previous steps, satisfy $A^{(4)}(t,0)=I$ for every $t\in [t_0,t_1]$. To this aim, we recall the construction introduced in \cite[\S 4]{NS}.

We define $\alpha:[t_0,t_1]\to \re^+$ and  $d:[t_0,t_1]\times \Omega\to \Omega$ as
\begin{align*}
\alpha(t) & := \sqrt{1-|\dot{s}^{(1)}(t)|^2}\,,
\\
d(t,x) & := \alpha(t) k_{\eta}(|x|) + (1-k_\eta(|x|)){c_1}\,,
\end{align*}
where $k_{\eta}$ is the following cut--off function:
\begin{equation}\label{def-k}
k_{\eta}(\tau):= \left\{
\begin{array}{lll}
1\quad & \hbox{if } 0 \leq \tau <\eta/2
\\
\Big(2\frac{\tau}{\eta} - 2 \Big)^2\Big(4\frac{\tau}{\eta} - 1\Big)\quad & \hbox{if } \eta/2 \leq \tau< \eta
\\
0 \quad & \hbox{if } \tau\geq \eta\,.
\end{array}
\right.
\end{equation}
Here $\eta$ is a positive parameter, whose precise value will be specified later, small enough such that $B_\eta(0)\subset \Omega$. Eventually, we set
$$
\P(t,x):= \left( \frac{x_1}{d(t,x)}, x_2\right)\,.
$$
For every $t\in [t_0,t_1]$, $\P$ defines a diffeomorphism of $\Omega$ into its dilation in the horizontal direction
$$
\Omega^{(4)}:= \left\{\left(\frac{x_1}{c_1},x_2\right)\ : \ x\in \Omega \right\}\,,
$$
which maps $0$ in $0$ and $\Gamma^{(3)}(t_0):=\Gamma^{(2)}(t_0)$ into a fixed set $\Gamma^{(4)}(t_0)$, horizontal near the origin. 
Accordingly, the tensor field $A^{(4)}$ associated to this transformation reads
\[
A^{(4)}(t,x) = \big[D\P\, A^{(3)} \, D\P^T - \dot\P\otimes \dot\P-  D\P \,\dot\Psi (\Psi^{-1}) \otimes \dot\P   - \dot \P \otimes D\P\, \dot\Psi (\Psi^{-1}) \big](\P^{-1}(t,x))\,.
\]

The properties of $A^{(4)}$ are gathered in the following

\begin{lemma}\label{lem-A4} There exists a constant $c_4>0$ such that, for every $ t\in [t_0,t_1]$ and $x\in \Omega^{(4)}$,
\begin{equation}\label{claim1}
(A^{(4)}(t,x) \xi)\cdot \xi \geq c_4|\xi|^2\,, \quad \forall \xi\in \mathbb R^2\,.
\end{equation}
Moreover, for every $t\in [t_0,t_1]$, there holds
\begin{equation}\label{claim2}
A^{(4)}(t,0) = I\,.
\end{equation}
Finally, there exists a vector field $W:\partial_N \Omega^{(4)} \cup \Gamma^{(4)}(t_0)\to \mathbb R^2$ such that, for every $ t\in[t_0,t_1]$ and $x\in \partial_N \Omega^{(4)} \cup \Gamma^{(4)}(t_0)$, 
\begin{equation}\label{normal}
(A^{(4)})^T(t,x) n(x) = W(x)\,,
\end{equation}
and $W(x)=n(x)=e_2$ in a neighborhood of the tip of $\Gamma^{(4)}(t_0)$.
\end{lemma}

\begin{proof}
Let $t\in [t_0,t_1]$ and $x\in \Omega^{(4)}$ be fixed. Setting $y:=\P^{-1}(t,x) \in \Omega$, we distinguish the three cases $|y|<\eta/2$, $\eta/2 \leq |y|\leq \eta$, and $|y|>\eta$, where $\eta$ is the constant introduced in \eqref{def-k}. 

Without loss of generality, up to take $\eta$ smaller, recalling \eqref{intorno}, we may assume that in $B_\eta(0)$
$$
D\Psi(\Psi^{-1}(t,y)) = I \,,\quad \dot{\Psi}(t,\Psi^{-1}(y)) = - \dot{s}^{(1)}(t)e_1\,,
$$
so that
$$
A^{(3)}(t,P^{-1}(t,x))= A^{(3)}(t,y) =  A^{(2)}(y)- |\dot{s}^{(1)}(t)|^2e_1\otimes e_1 \,.  
$$
Moreover, we take $\eta<r$ with $r$ associated to $\e = c_1^2/2$ as in \eqref{ell-A2}, so that the ellipticity constant of $A^{(2)}$ in $B_{\eta}(0)$ is $(1-c_1^2/2)$.

\smallskip

If $|y|<\eta/2$ we have
$$
D \P(t,y) = \left( 
\begin{array}{ccc}
\frac{1}{\alpha(t)} & 0 
\\ 0  & 1 
\end{array}
\right) \,,\quad \dot{\P}(t,y) =  \left( 
\begin{array}{ccc}
-y_1\frac{\dot{\alpha}(t)}{\alpha^2(t)}
\\ 0
\end{array}
\right)\,,
$$
thus
\begin{equation}\notag
A^{(4)}(t,x)=  \left( 
\begin{array}{ccc}
\frac{1}{\alpha(t)} & 0 
\\ 0  & 1 
\end{array}
\right) A^{(2)}(y) \left( 
\begin{array}{ccc}
\frac{1}{\alpha(t)} & 0 
\\ 0  & 1 
\end{array}
\right) -  \left( 
\begin{array}{ccc}
 \frac{ |\dot{s}^{(1)}(t)|^2}{\alpha(t)^2}  + y_1  \frac{2 \dot{s}^{(1)}(t)\dot{\alpha}(t)}{\alpha^3(t)} + y_1^2\frac{\dot{\alpha}^2(t)}{\alpha^4(t)}
\ & 0 
\\ 0  & 0
\end{array}
\right)\,.
\end{equation}
Since $P^{-1}(t,0)=0$ and $A^{(2)}(0)=I$, we immediately get \eqref{claim2}.
For $\xi$ arbitrary vector of $\mathbb R^2$, we have
$$
(A^{(4)}(t,x) \xi) \cdot \xi  \geq \left[\frac{1-c_1^2/2-  |\dot{s}^{(1)}(t)|^2}{\alpha^2(t)}  - 2 y_1  \frac{ \dot{s}^{(1)}(t)\dot{\alpha}(t)}{\alpha^3(t)} - y_1^2\frac{\dot{\alpha}^2(t)}{\alpha^4(t)}
\right] \xi_1^2 + (1-c_1^2/2) \xi _2^2\,.
$$
In view of the bounds \eqref{ipotesi-s}, \eqref{def-c2}, and \eqref{bound-s1}, we get
$$
|\dot{\alpha}(t)|\leq \frac{c_2}{c_1}\,,\quad  c_1 \leq |\alpha(t)| \leq 1\,,
$$
in particular
$$
(A^{(4)}(t,x) \xi) \cdot \xi  \geq \left(\frac{c_1^2}{2} - 2 \eta\frac{c_2}{c_1^4} - \eta^2\frac{c_2^2}{c_1^6}
\right) \xi_1^2 + \frac{\xi _2^2}{2}\,.
$$
The coefficient of $\xi_1$ is bounded from below, provided that $\eta$ is small enough. This gives the statement \eqref{claim1} for $y\in B_{\eta/2}(0)$.

\smallskip

Let now $\eta /2 < |y|<\eta$. In this case we have 
$$
D \P(t,y) = \frac{1}{d^2}\left( 
\begin{array}{ccc}
d - y_1 \partial_1 d \ &  - y_1 \partial_2 d
\\ 0  & d^2 
\end{array}
\right) \,,\quad \dot{\P}(t,y) =  \frac{1}{d^2}\left( 
\begin{array}{ccc}
 - y_1 \partial_t d
\\ 0
\end{array}
\right)\,.
$$
Again exploiting the ellipticity of $A^{(2)}$ with constant $(1-c_1^2/2)\geq \frac12 $ and setting 
$$
m:= y_1^2 (\partial_t d)^2 + 2 y_1  \dot{s}^{(1)}(t)  (\partial_t d)  (d - y_1 \partial_1 d)\,,\quad p:=  (d-y_1 \partial_1 d)   \,,\quad q:=  - y_1 \partial_2 d   \,,
$$
we get
\begin{align}
(A^{(4)}(t,x) \xi) \cdot \xi  & \geq  \frac12 \| DP^{T}(t,y)\xi  \|^2 - \frac{m}{d^4} \xi_1^2  = \frac{1}{2 d^4} \left[ ( p^2 + q^2 - 2 m ) \xi_1^2 + 2 q d^2 \xi_1\xi_2 + d^4 \xi_2^2  \right]\notag
\\ & \geq \frac{1}{2} \left[  p^2  - \left ( \frac{1}{\varepsilon}  -1 \right) q^2 - 2 |m|  \right]  \xi_1^2  + \frac{1}{2} (1-\varepsilon) \xi_2^2 \,,\label{coerc-A4}
\end{align}
where in the last inequality we have have used $d\leq 1$ and the Young's inequality, with $0 < \varepsilon < 1$, whose precise value will be fixed later.
Let us prove that, if $\eta$ and $\varepsilon$ are well chosen, the coercivity of $A^{(4)}$ is guaranteed. The identities
$$
\nabla_y d (t,y) = (\alpha(t)-c_1) \frac{y}{|y|} k_\eta'(|y|)\,,\quad  \partial_t d (t,y) =  - \frac{\dot{s}^{(1)}(t) \ddot{s}^{(1)}(t) k_\eta(|y|)}{\alpha(t)} \,,
$$
together with the bounds
$$
0\leq k_\eta\leq 1\,,\quad c_1 \leq d \leq \alpha \leq 1 \,\quad -\frac{3}{\eta}\leq k_\eta'\leq 0\,,
$$
give
\begin{align*}
& \frac{1}{d^4}\geq 1\,,
\\
& p=d+ \frac{y_1^2}{|y|}(\alpha-c_1)|k'_\eta(|y|)| \geq d \geq c_1\,, 
\\
&  q^2 =  (\alpha -c_1)^2 \frac{y_1^2 y_2^2}{|y|^2}(k'_\eta(|y|))^2    \leq  9  (1-c_1)^2 \,,
\\
&  |m| \leq  \frac{ 42 c_2  (1-c_1^2) }{c_1} \eta + \frac{c_2^2  (1-c_1^2) }{c_1^2}\eta^2\,.
\end{align*}
Inserting these estimates into \eqref{coerc-A4}, we infer that
$$
(A^{(4)}(t,x) \xi) \cdot \xi \geq \left[ \frac{c_1^2}{2} - \frac{9}{2} \left( \frac{1}{\varepsilon} - 1  \right) (1-c_1)^2 -  \frac{ 42 c_2  (1-c_1^2) }{c_1} \eta -  \frac{c_2^2  (1-c_1^2) }{c_1^2}\eta^2\right] \xi_1^2 + \frac{ 1-\varepsilon}{2} \xi_2^2\,.
$$
Taking
$$
\varepsilon = \frac{ 9 (1-c_1)^2  }{c_1^2/2 + 9 (1-c_1)^2 }\in (0,1)
$$
we have
$$
 \frac{c_1^2}{2} - \frac{9}{2} \left( \frac{1}{\varepsilon} - 1 \right) (1-c_1)^2 = \frac{c_1^2}{4}\,.
$$
Thus, taking $\eta$ small enough, we obtain the desired coercivity of $A^{(4)}$.

\smallskip

Finally, if $|y|>\eta$ we have
$$
D \P(t,y) = \left( 
\begin{array}{ccc}
\frac{1}{c_1} & 0 
\\ 0  & 1 
\end{array}
\right) \,,\quad \dot{\P}(t,y) =0\,,
$$
and condition \eqref{claim1} is readily satisfied in view of the ellipticity of $A^{(3)}$.

\medskip

The assertion \eqref{normal} is clearly verified for $A^{(2)}$: the tensor field does not depend on time and equals to the identity on the fracture, in a neighborhood of the origin. The last diffeomorphisms $\Psi$ and $\P$ both act in a neighborhood of the origin modifying the set only in the horizontal component; in particular they don't modify the normal to the fracture in a neighborhood of the origin. As for the external boundary, $\Psi$ is the identity and $P$ acts as a constant dilation, so that 
$$
W(x)=\left( 
\begin{array}{ccc}
\frac{1}{c_1} & 0 
\\ 0  & 1 \end{array}\right) A^{(2)}(c_1 x_1, x_2)\left( 
\begin{array}{ccc}
\frac{1}{c_1} & 0 
\\ 0  & 1 \end{array}\right)n(x) \quad \hbox{on }\partial_N \Omega^{(4)}\,.
$$
\end{proof}

\begin{remark}\label{differenze}
The idea of the proof of Lemma \ref{lem-A4} is taken from \cite[Lemma 4.1]{NS}. Let us underline the main differences: in \cite{NS} the authors deal with the identity matrix as starting tensor field (here instead we have $A^{(3)}$) and consider only the dynamics for which the acceleration of the tip is bounded by a precise constant depending on $c_1$ (in place of our bound $c_2$, not fixed a priori). We also point out that in \cite{NS} the study of the ellipticity of the transformed tensor field, in the annulus $\eta/2 < |y| < \eta$, is carried out forgetting the coefficients out of the diagonal.
\end{remark}

\begin{remark}\label{remrho}
In our construction, a control on the maximal amplitude $\rho$ of the time interval $[t_0,t_1]$ is needed only in Step 2: roughly speaking, in order to straighten the set $\Gamma^{(1)}(t_1)\setminus \Gamma^{(1)}(t_0)$ and to remain inside $\Omega$, we need to have enough room. A sufficient condition is that the length of the set, which is at most $\rho \max_{t\in[0,T]} \dot{s}^{(1)}(t)$, has to be less than or equal to the distance of the crack-tip $\gamma^{(1)}(s^{(1)}(t))$ from the boundary $\partial \Omega$, which is, thanks to the assumption $\Gamma^{(1)}(T)\setminus \Gamma^{(1)}(0)\subset \subset\Omega$, bounded from below by a positive constant. Notice that if we considered also a further diffeomorphism which is the identity in a neighborhood of $\Gamma^{(1)}(T)\setminus \Gamma^{(1)}(0)$ and stretches $\Omega$ near the boundary, then our results could be stated for every time $t\in[0,T]$.
\end{remark}

\section{Proof of the representation result}\label{sec-dec}
 In this section we derive the decomposition result \eqref{decomposition} locally in time, namely in a time interval $[t_0,t_1]$ small enough (see \S \ref{ssec-cov} and Remark \ref{remrho}). Finally, in \S \ref{maicol}, we give a global representation of $u$, valid in the whole time interval $[0,T]$.

\subsection{Preliminaries on semigroup theory}
Here we recall some classical facts of semigroup theory. Standard references on the subject are the books \cite{P} and \cite{Kato}.

\medskip

Let $X$ be a Banach space and $\mathcal A(t):\mathrm{D}(\mathcal A(t))\subset X \to X$ a differential operator. Consider the evolution problem 
\begin{equation}\label{system-V}
\partial_t V(t) + \mathcal A (t) V(t) = {G}(t)\,,
\end{equation}
with initial condition $V(0)=V_0$ (the boundary conditions are encoded in the function space $X$).

\begin{definition}\label{CD}
A triplet $\{\mathcal A;X,Y\}$ consisting of a family $\mathcal A=\{\mathcal A(t)\,,\ t\in [0,T]\}$ and a pair of real separable Banach spaces $X$ and $Y$ is called a {\it constant domain system} if the following conditions hold:
\begin{itemize}
\item[i)] the space $Y$ is embedded continuously and densely in $X$;
\item[ii)] for every $t$ the operator $\mathcal A(t)$ is linear and has constant domain $\mathrm{D}(\mathcal A(t))\equiv Y$;
\item[iii)] the family $\mathcal A$ is a stable family of (negative) generators of strongly continuous semigroups on $X$;
\item[iv)] the operator $\partial_t\mathcal A$ is essentially bounded from $[0,T]$ to the space of linear functionals from $Y$ to $X$.
\end{itemize}
\end{definition}

\begin{theorem}\label{thm-kato}
Let $\{\mathcal A;X,Y\}$ form a constant domain system. Let $V^0\in Y$ and {$G\in \mathrm{Lip}([0,T];X)$}. Then there exists a unique solution $V \in C([0,T]; Y) \cap C^1([0,T]; X)$ of \eqref{system-V} with $V(0)=V^0$.
\end{theorem}

\subsection{Local representation result in the cylindrical domain}
The chain of transformations introduced in \S \ref{ssec-cov} defines the family of time-dependent diffeomorphisms 
\begin{equation}\label{def-Phi}
\Phi(t):=  \P (t) \circ \Psi (t) \circ \L \circ \C \,,\quad \Phi(t): \overline{\Omega}\to \overline{\Omega}^{(4)}\,, 
\end{equation}
which map $\Gamma$ into $\Gamma^{(4)}$, $\Gamma(t)$ into $\Gamma^{(4)}(t_0)$ for every $t\in [t_0,t_1]$, and $\partial \Omega$ into $\partial \Omega^{(4)}$. More precisely, the Dirchlet part $\partial_D \Omega$ is mapped into $\partial_D\Omega^{(4)}:= \{(\Lambda_1(x)/c_2, \Lambda_2(x))\ :\ x \in \partial_D \Omega\}$, the Neumann one $\partial_N \Omega$ into  $\partial_N\Omega^{(4)}:= \{(\Lambda_1(x)/c_2, \Lambda_2(x))\ :\ x \in \partial_N \Omega\}$.  For the sake of clarity, we denote by $x$ the variables in $\Omega$ and by $y$ the new variables in $\Omega^{(4)}$.

Looking for a solution $u$ to \eqref{eq-1} is equivalent to look for $v:= u\circ \Phi^{-1}$, solution to 
\begin{equation}\label{eq-v}
\ddot{v}(t) - \div (A^{(4)} \nabla v (t) ) + {p}(t) \cdot \nabla v(t) - 2 {q}(t) \cdot \nabla \dot{v}(t) = {g}(t) \quad \hbox{in }\Omega^{(4)}\setminus \Gamma^{(4)}(t_0)\,,
\end{equation}
supplemented by the boundary conditions
\begin{equation}\label{bc-v}
v=0 \quad \hbox{on }\partial_D \Omega^{(4)}\,,\quad \partial_W v = 0 \quad \hbox{on }\partial_N \Omega^{(4)} \cup\Gamma^{(4)}(t_0)\,,
\end{equation} 
and by suitable initial conditions. Here $W$ is the vector field introduced in \eqref{normal} - Lemma \ref{lem-A4}, and (see also  \cite{DM-Luc})
\begin{align*}
{p}(t,y) & := - [A^{(4)}(t,y) \nabla(\mathrm{det} D\Phi^{-1}(t,y))+ \partial_t ({q}(t,y) \mathrm{det} D\Phi^{-1}(t,y)    )   ]\mathrm{det} D \Phi(t,\Phi^{-1}(t,y))\,,
\\
{q}(t,y) & := - \dot{\Phi} (t, \Phi^{-1}(t,y))\,,
\\
{g}(t,y) & := f(t, \Phi^{-1}(t,y))\,.
\end{align*}

The characterization of $u$ will follow from that of $v$, slightly easier to be derived. As already pointed out in the Introduction, the advantages in dealing with problem \eqref{eq-v} are essentially 3: first of all, the domain is cylindrical and constant in time; then, the fracture set is straight near the tip; finally, even if the coefficients depend on space and time, the principal part of the spatial differential operator is constant at the crack-tip.

Before stating the result, we define
\begin{align*}
& H^1_D(\Omega^{(4)}\setminus \Gamma^{(4)}(t_0)) :=\{v \in H^1(\Omega^{(4)}\setminus \Gamma^{(4)}(t_0)) \ :\ v = 0 \ \hbox{on }\partial_D \Omega^{(4)}\}\,, 
\\
& {\mathcal H} := \{ v \in H^2(\Omega^{(4)}\setminus \Gamma^{(4)}(t_0)) \ : \eqref{bc-v}\ \hbox{hold true} \} \oplus \{k \zeta S\ :\ k\in \mathbb R\}\,,
\end{align*}
where $\zeta$ is a cut--off function whose support contains the origin and
\begin{equation}\label{def-S}
S(y):=Im (\sqrt{y_1 + i y_2})\,.
\end{equation}

\begin{figure}[h]                                             
\begin{center}                                                
{\includegraphics[height=1.5truecm] {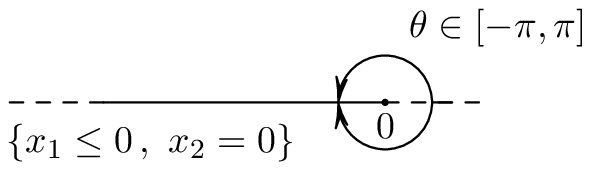}}                                                
\end{center}                                                  
\caption{{\it In polar coordinates, the function $S$ reads $S(r,\theta)=r^{1/2} \sin (\theta/2)$, where $r$ is the distance from the origin and $\theta\in [-\pi,\pi]$ is the angle which has a discontinuity on the horizontal half line $\{x_1\leq 0\}$.}}\label{S} 
\end{figure}

\begin{proposition}\label{prop-v}
Take $v^0 \in \mathcal H$, $v^1 \in H^1_D(\Omega^{(4)}\setminus \Gamma^{(4)}(t_0))$, and $g\in \mathrm{Lip}([t_0,t_1];L^2(\Omega^{(4)}))$. Then there exists a unique solution $v$ to \eqref{eq-v}-\eqref{bc-v} with $v(t_0)=v^0$, $\dot{v}(t_0)=v^1$ in the class
$$
v\in C([t_0,t_1];  {\mathcal H}) \cap C^1([t_0,t_1]; H^1_D(\Omega^{(4)}\setminus \Gamma^{(4)}(t_0))) \cap C^2([t_0,t_1];L^2(\Omega^{(4)}))\,.
$$
\end{proposition}

\begin{proof}
Once we show that the triplet $\{\mathcal A; X; Y\}$ defined by
\begin{align*}
& \mathcal A(t):= \left(\begin{array}{ccc}
0 & -1
\\
-\div (A^{(4)}(t) \nabla (\cdot)) +  p(t)\cdot \nabla(\cdot) &  - 2 q(t) \cdot \nabla (\cdot)
\end{array}\right) \,, 
\\
& X:=H^1_D(\Omega^{(4)}\setminus \Gamma^{(4)}(t_0)) \times L^2(\Omega^{(4)})\,,
\\
&
Y:= \mathcal H \times H^1_D(\Omega^{(4)}\setminus \Gamma^{(4)}(t_0))\,,
\end{align*}
is a constant domain system in $[t_0,t_1]$ (cf. Definition \ref{CD}), we are done. Indeed, we are in a position to apply Theorem \ref{thm-kato} with 
$$
G(t):= \left(\begin{array}{c}0 \\  g(t) \end{array}\right)\,,
$$
and the searched $v$ is the second component of the solution $V$ to \eqref{system-V}.

The detailed proof of properties (i)-(iv) in Definition \ref{CD} can be found in \cite[Theorem 4.7]{NS}, with the appropriate modifications (see Remark \ref{differenze}). Here we limit ourselves to list the main ingredients.

First of all, the domain of $\div (A^{(4)}(t)\nabla (\cdot))$ is constant in time: in view  \eqref{claim2}, its principal part, evaluated at the crack tip, is the Laplace operator for every $t$, thus the domain of $\div (A^{(4)}(t)\nabla (\cdot))$ can be decomposed as the sum $\{ v \in H^2(\Omega^{(4)}\setminus \Gamma^{(4)}(t_0)) \ :\ \eqref{bc-v}\ \hbox{holds true} \} \oplus \{\zeta S\}=: \mathcal H$ (cf. \cite[Theorem 5.2.7]{Gr}); moreover,  in view of \eqref{normal}, the boundary conditions \eqref{bc-v} do not depend on time. 

Other key points are the equi coercivity in time of the bilinear form
$$
(w_0,w_1)\mapsto (A^{(4)}(t) \nabla w_0)\cdot \nabla w_1
$$ 
in $H^1_D(\Omega^{(4)}\setminus \Gamma^{(4)}(t_0))$, guaranteed by \eqref{claim1}, and the property
$$
 \int_{\Omega^{(4)}\setminus \Gamma^{(4)}(t_0)}   ({q}(t)\cdot \nabla \varphi)\, \varphi\, \de y = -\frac12 \int_{\Omega^{(4)}\setminus \Gamma^{(4)}(t_0)} \varphi^2 \div {q}(t)\, \de y\,,
$$
valid for every $\varphi \in H^1_D(\Omega^{(4)}\setminus \Gamma^{(4)}(t_0))$.

Finally, the needed continuity of the differential operator is ensured by the following regularity properties of the coefficients: for every $i,j,k\in \{1,2\}$,
\begin{align*}
& A_{i,j}^{(4)}(t) \in C^0(\Omega^{(4)})   \quad \forall t\in [t_0,t_1]
\\
&
A_{i,j}^{(4)}\,,\ p_i\,,\, q_i \in \mathrm{Lip}([t_0,t_1]; L^\infty(\Omega^{(4)}))\,,
\\
& \| \partial_k  A_{i,j}^{(4)}(t )\|_{ L^\infty(\Omega^{(4)})}\,,\  \| \div q(t)\|_{ L^\infty(\Omega^{(4)})} \leq C\,,
\end{align*}
for a suitable constant $C>0$ independent of $t$.
\end{proof}

\subsection{Local representation result in the time-dependent domain}
We are now in a position to prove the following representation result for $u$.
\begin{theorem}\label{thm-decomp} Let $f\in C^0([t_0,t_1];H^1(\Omega))\cap {\mathrm{Lip}([t_0,t_1];L^2(\Omega))}$ and  $\zeta(t)$ be a $C^2$ (in time) family of cut--off functions with support in a neighborhood of $\gamma(s(t))$. Consider $u^0$ and $u^1$ of the form
\begin{align*}
& u^0 - k^0S(\Phi(0,\cdot)) \in  H^2(\Omega\setminus \Gamma(t_0))\,,\quad 
\\
& u^1 - \nabla u^0\cdot \left (D\Phi^{-1}(0)\dot{\Phi}(0)\right) \in H^1_D(\Omega \setminus \Gamma(t_0))\,,
\end{align*}
 $u^0$ satisfying the boundary conditions \eqref{bc-u} and $k^0\in \mathbb R$. Then there exists a unique solution to \eqref{eq-1}-\eqref{bc-u} with initial conditions $u(t_0)=u^0$, $\dot{u}(t_0)=u^1$ of the form
\begin{equation}\label{decomp1}
u(t,x)=u^{R}(t,x) + k(t) \zeta(t,x) S(\Phi(t,x))\,,
\end{equation}
where $k$ is a $C^2$ function such that $k(t_0)=k^0$. Moreover, 
$$
u^R \in C^2([t_0,t_1]; L^2(\Omega))\,,\quad \nabla u^R \in C^1([t_0,t_1]; L^2(\Omega;\mathbb R^2))\,,\quad \nabla^2 u^R\in C^0([t_0,t_1];L^2(\Omega;\mathbb R^{2\times 2}))\,,
$$
and $u^{R}(t) \in H^2(\Omega\setminus \Gamma(t))$ for every $t\in [t_0,t_1]$.
\end{theorem}

\begin{remark}\label{rem-initcond} Notice that the equality $u(t,x)=v(t,\Phi(t,x))$ implies that
$$
u^0 = v^0(\Phi(t_0))\,,\quad u^1 = v^1(\Phi(t_0)) + \nabla v^0(\Phi(t_0))\cdot \dot{\Phi}(t_0)\,.
$$
The last term reads $\dot\Phi(t_0) = \dot{P}(t_0,\Psi\circ \L\circ\C) + DP (t_0,\Psi\circ \L\circ\C) \cdot \dot\Psi (t_0,\L\circ\C)$. 
A priori, $\nabla v^0$ is just in $L^2$ in a neighborhood of the origin and its gradient behaves like $|y|^{-3/2}$; nevertheless, since $\dot{\P}(t,y) \sim(y_1,0)$, we recover the $L^2$ integrability of the gradient of $\nabla v^0(\Phi(t_0)) \cdot \dot{\P}(t_0,\Psi\circ \L\circ\C)$. The same reasoning does not apply for the term $\nabla v^0 (\Phi(t_0))\cdot \left(DP (t_0,\Psi\circ \L\circ\C) \cdot \dot\Psi (t_0,\L\circ\C)\right)$, since the singularity of $\nabla v^0$ in a neighborhood of the orgin is not compensated by $DP \dot\Psi$.
Therefore we are not free to take $u^1\in H^1_D(\Omega\setminus\Gamma(t_0))$ (as, on the contrary, is done in \cite{NS}).
\end{remark}

\begin{remark}
Note that the solution $u$ to \eqref{eq-1}-\eqref{bc-u} displays a singularity only at the crack-tip. Clearly, the fracture is responsible for this lack of regularity. On the other hand, the Dirichlet-Neumann boundary conditions do not produce any further singularity, due to the compatible initial data chosen.
\end{remark}

\subsection{Global representation result in the time-dependent domain}\label{maicol}

We conclude the section by showing an alternative representation formula which can be expressed for every time. This is done providing another expression for the singular function, as in \cite{Laz-Toa}, whose computation does not require to straighten the crack. To simplify the notation we reduce ourselves to the case $A = I$, so that the diffeomorphism $\chi$ coincides with the identity.

The chosen singular part of the solution to problem \eqref{eq-1}-\eqref{bc-u} is a suitable raparametrization of the function $S$ introduced in \eqref{def-S}. More precisely, fixed $t_0,t_1\in[0,T]$ with $0<t_1-t_0<\rho$, for every $t\in[t_0,t_1]$ and $x$ in a neighborhood of $r(t):=\gamma(s(t))$, the singular part reads
\begin{equation}\label{nohatS}
S\left(\frac{\Lambda_1(x)-(s(t)-s(t_0))}{\sqrt{1-|\dot s(t)|^2}},\Lambda_2(x)\right)\,.
\end{equation}
To compute \eqref{nohatS} it is necessary to know the expression of $\Lambda$, which is explicit only for small time and locally in space. We hence provide a more explicit formula for the singular part, which has also the advantage of being defined for every time: for every $t\in [0,T]$ we set
\begin{equation}\label{hatS}
\hat S(t,x):=Im\left(\sqrt{\frac{\left(x-r(t)\right)\cdot \gamma'(s(t))}{\sqrt{1-|\dot s(t)|^2}}+i\left(x-r(t)\right)\cdot n(s(t))}\right)\,,
\end{equation}
where $n(\sigma)\perp\gamma'(\sigma)$ and $\hat S(t)$ is given by the unique continuous determination of the complex square function such that in $x=r(t)+\sqrt{1-|\dot s(t)|^2}\gamma'(s(t))$ takes value 1 and its discontinuity set lies on $\Gamma(t)$. Roughly speaking, if we forget the term $\sqrt{1-|\dot s(t)|^2}$, the function \eqref{hatS} is the determination of $Im(\sqrt{y_1 + iy_2})$ in the orthonormal system with center $\gamma(s(t))$ and axes $\gamma'(s(t))$ and $n(s(t))$.

\begin{figure}[h]                                             
\begin{center}                                                
{\includegraphics[height=3.5truecm] {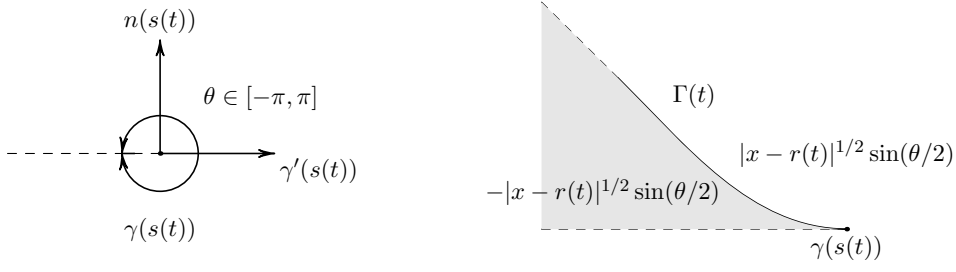}}                                                
\end{center}                                                  
\caption{{\it A possible choice of determination of $Im(\sqrt{y_1+iy_2})$, with $\Gamma(t)$ as discontinuity set.}}\label{S2} 
\end{figure}                                                  

For every $t\in[0,T]$ let $R(t)\in SO(2)^+$ be the matrix that rotates the orthonormal system with axes $\gamma'(s(t))$ and $n(s(t))$ in the one with axes $e_1$ and $e_2$. Thanks to our construction of $\Lambda$, and in particular to \eqref{Lambda}, the matrix $R(t)$ coincides with $D\Lambda(r(t))$ in $[t_0,t_1]$. By setting
\begin{align*}
&L(t):=\left( 
\begin{array}{ccc}
\frac{1}{\sqrt{1-|s(t)|^2}} & 0 
\\ 0  & 1 
\end{array}
\right)
\,,\quad\tilde \Phi(t,x):=L(t)R(t)(x-r(t))\,,\\
&\tilde\Omega(t):=\tilde\Phi(t,\Omega)\,,\quad\tilde\Gamma(t):=\tilde\Phi(t,\Gamma(t))\,,
\end{align*}
we may also write $\hat S(t,x)=\tilde S(t,\tilde\Phi(t,x))$, where $\tilde S(t,\cdot)$ is given by the continuous determination of $Im(\sqrt{y_1+iy_2})$ in $\tilde\Omega(t)\setminus\tilde\Gamma(t)$ such that in $y=(1,0)$ takes the value 1.

\begin{lemma}\label{lemregxw}
Under the same assumptions of Theorem~\eqref{thm-decomp}, the function $w(t):=S(\Phi(t))-\hat S(t)$ belongs to $H^2(\Omega\setminus\Gamma(t))$ for every $t\in[t_0,t_1]$.
\end{lemma}

\begin{proof}
Let us fix $t\in[t_0,t_1]$.  The function $w(t)$ is of class {$C^2$} in $\Omega\setminus\Gamma(t)$ and it belongs to $H^1(\Omega\setminus\Gamma(t))\cap H^2((\Omega\setminus\Gamma(t))\setminus B_\varepsilon(r(t)))$ for every $\varepsilon>0$. Hence it remains to prove the $L^2$-integrability of its second spatial derivatives in $B_\varepsilon(r(t))$. For every $i,j\in\{1,2\}$ we have
\begin{align*}
\partial^2_{ji} w(t)&=\sum_h(\partial_hS(\Phi(t))\partial^2_{ji} \Phi_h(t)-\partial_h \tilde S(t,\tilde\Phi(t))\partial^2_{ji}\tilde\Phi_h(t))\\
&+\sum_{h,k}(\partial^2_{hk}S(\Phi(t))\partial_j \Phi_k(t)\partial_i \Phi_h(t)-\partial^2_{hk} \tilde S(t,\tilde\Phi(t))\partial_j\tilde\Phi_k(t)\partial_i\tilde\Phi_h(t))=:I_1(t)+I_2(t)\,,
\end{align*}
where $\Phi_i(t)$ and $\tilde\Phi_i(t)$ are the $i$--th components of $\Phi(t)$ and $\tilde\Phi(t)$, respectively.
	
Notice that $\nabla S(\Phi(t)), \nabla \tilde S(t,\tilde\Phi(t))\in L^2(\Omega\setminus\Gamma(t);\mathbb R^2)$, while $D^2 \Phi(t)$ and $D^2 \tilde\Phi(t)$ are uniformly bounded in $\Omega$. Therefore $I_1(t)\in L^2(\Omega\setminus\Gamma(t))$ and in particular there exists a positive constant $C$, independent of $t$, such that
\begin{equation*}
|I_1(t,x)|\le C|x-r(t)|^{-\frac{1}{2}}\quad\text{for every }x\in B_\varepsilon(r(t))\setminus\Gamma(t)\,,
\end{equation*}
provided that $\varepsilon>0$ is small enough.
	
As for $I_2(t)$, we estimate it from above as
\begin{equation}\label{vbe}
\begin{aligned}
|I_2(t)|&\le \sum_{h,k}|\partial^2_{hk}S(\Phi(t))-\partial^2_{hk} \tilde S(t,\tilde\Phi(t))||\partial_j\tilde\Phi_k(t)||\partial_i\tilde\Phi_h(t)|\\
&+\sum_{h,k}|\partial^2_{hk}S(\Phi(t))||\partial_j \Phi_k(t)\partial_i \Phi_k(t)-\partial_j\tilde\Phi_k(t)\partial_i\tilde\Phi_h(t)|\,.
\end{aligned}
\end{equation}
Let us study the right--hand side of \eqref{vbe}. By choosing $\varepsilon$ small enough and by using the definitions of $\Phi(t)$ and $\tilde\Phi(t)$, we deduce that for every $x\in B_\varepsilon(r(t))$
\begin{equation}\label{Sest1}
\begin{aligned}
|\partial_j \Phi_k(t,x)&\partial_i \Phi_h(t,x)-\partial_j\tilde\Phi_k(t,x)\partial_i\tilde\Phi_h(t,x)|\le \frac{2}{c_1^2} \Vert D\Lambda\Vert_\infty \Vert D^2\Lambda\Vert_\infty |x-r(t)|\,,
\end{aligned}
\end{equation}
since $\Vert D\Phi (t)\Vert_\infty$,$\Vert D\tilde\Phi(t)\Vert_\infty\le\Vert D\Lambda\Vert_\infty/c_1$ and 
\begin{equation*}
|D \Phi(t,x)-D \tilde\Phi(t,x)| \le \frac{1}{c_1} |D\Lambda(x)-R(t)| \le \frac{1}{c_1} \Vert D^2\Lambda\Vert_\infty |x-r(t)|\,.
\end{equation*} 
Moreover, the function $S$ satisfies $|\nabla^2 S(y)|\le M|y|^{-\frac{3}{2}}$ in $\Omega^{(4)}\setminus\Gamma^{(4)}(t_0)$ for a positive constant $M$, while $\Lambda$ is invertible and $|P(t,x)|\ge |x|$. This allows us to conclude that
\begin{equation}\label{Sest2}
|\partial^2_{hk}S(\Phi(t,x))|\le M\Vert D\Lambda^{-1}\Vert_\infty^{\frac{3}{2}}|x-r(t)|^{-\frac{3}{2}}\quad\text{for every }x\in B_\varepsilon(r(t))\setminus\Gamma(t)\,.
\end{equation}
	
Regarding the second term in the right--hand side of \eqref{vbe}, we fix $x\in B_\varepsilon(r(t))$ and we consider the segment $[\Phi(t,x),\tilde\Phi(t,x)]:=\{\lambda\Phi(t,x)+(1-\lambda)\tilde\Phi(t,x)\ :\ \lambda\in[0,1]\}$ and the function $d(t,x):=\dist([\Phi(t,x),\tilde\Phi(t,x)],0)$. We claim that we can choose $\varepsilon>0$ so small that
\begin{equation}\label{distest}
d(t,x)\ge \frac{1}{2}|x-r(t)|\quad\text{for every }x\in B_\varepsilon(r(t))\,.
\end{equation}
Indeed let $y\in [\Phi(t,x),\tilde\Phi(t,x)]$ be such that $|y|=d(t,x)$, then
\begin{equation*}
|\tilde\Phi(t,x)|\le |y|+|\tilde\Phi(t,x)-y|\le |y|+|\tilde\Phi(t,x)- \Phi(t,x)|\,.
\end{equation*}
Since $|P(t,x)|\ge |x|$ and $R(t)$ is a rotation, for $\varepsilon$ small we deduce that $|\tilde\Phi(t,x)|\ge |x-r(t)|$. On the other hand, by the Lagrange Theorem there exists $z=z(t,x)\in B_\varepsilon(r(t))$ such that
\begin{align*}
\Phi(t,x)&= \Phi(t,r(t))+ D\Phi(t,r(t))(x-r(t))+D^2\Phi(t,z)(x-r(t))\cdot(x-r(t))\\
&=\tilde\Phi(t,x)+D^2\Phi(t,z)(x-r(t))\cdot(x-r(t))\,.
\end{align*}
Hence we derive the estimate
\begin{equation}\label{Phiest}
|\Phi(t,x)- \tilde\Phi(t,x)|\le \frac{1}{c_1}\Vert D^2 \Lambda\Vert_\infty|x-r(t)|^2\quad\text{for every }x\in B_\varepsilon(r(t))\,,
\end{equation}
which implies 
\begin{align*}
d(t,x)&\ge |x-r(t)|-\frac{1}{c_1}\Vert D^2 \Lambda\Vert_\infty|x-r(t)|^2\quad\text{for every }x\in B_\varepsilon(r(t))\,.
\end{align*}
In particular we obtain \eqref{distest} by choosing $\varepsilon<c_1/(2\Vert D^2 \Lambda\Vert_\infty)$. Notice that $\varepsilon$ does not depend on $t\in[t_0,t_1]$.
	
Let us now fix $x\in B_\varepsilon(r(t))\setminus\Gamma(t)$. Thanks to our construction of $\Phi$ and $\tilde\Phi$, it is possible to find two other determinations $S^{\pm}(t)$ of $Im(\sqrt{y_1+iy_2})$ in $\mathbb R^2$ such that their discontinuity sets $\Gamma^{\pm}(t)$ do not intersect the segment $[\Phi(t,x),\tilde\Phi(t,x)]$, which is far way from 0. Moreover, we choose them in such a way that $S^+(t)$ is positive along $\{(x_1,0)\ :\ x_1\le 0\}$, while $S^-(t)$ is negative, and $S(\Phi(t,x))=S^{\pm}(t,\Phi(t,x))$ if and only if $\tilde S(t,\tilde\Phi(t,x))=S^{\pm}(t,\tilde\Phi(t,x))$; notice that $|\nabla^3S^{\pm}(t,y)|\le M|y|^{-\frac{5}{2}}$ for a positive constant $M$ and for every $y\in \mathbb R^2\setminus\Gamma^\pm(t)$. By using the Lagrange Theorem, \eqref{distest}, and \eqref{Phiest}, we deduce that 
\begin{equation}\label{Sest3}
\begin{aligned}
|\partial^2_{hk}S(\Phi(t,x))&-\partial^2_{hk} \tilde S(t,\tilde\Phi(t,x))|=|\partial^2_{hk} S^{\pm}(t,\Phi(t,x))-\partial^2_{hk} S^{\pm}(t,\tilde\Phi(t,x))|\\ 
&\le |\nabla^3S^{\pm}(t,z)||\Phi(t,x)-\tilde\Phi(t,x)|\le \frac{M}{c_1}\Vert D^2 \Lambda\Vert_\infty|d(t,x)|^{-\frac{5}{2}}|x-r(t)|^2\\
&\le \frac{4\sqrt{2}M}{c_1}\Vert D^2 \Lambda\Vert_\infty|x-r(t)|^{-\frac{1}{2}}\,,
\end{aligned}
\end{equation}
where $z=z(t,x)\in[\Phi(t,x),\tilde\Phi(t,x)]$. Hence, by combining \eqref{vbe} with \eqref{Sest1}, \eqref{Sest2}, and \eqref{Sest3}, we obtain the existence of a positive constant $C$ such that
\begin{equation*}
|I_2(t,x)|\le C|x-r(t)|^{-\frac{1}{2}}\quad\text{for every }x\in B_\varepsilon(r(t))\setminus\Gamma(t)\,.
\end{equation*}
In particular we get the following bound for $\nabla^2w$:
\begin{equation}\label{nabla2w}
|\nabla^2 w(t,x)|\le C|x-r(t)|^{-\frac{1}{2}}\quad\text{for every }x\in B_\varepsilon(r(t))\setminus\Gamma(t)\,,
\end{equation}
and consequently $w(t)\in H^2(\Omega\setminus\Gamma(t))$ for every $t\in [t_0,t_1]$. 
\end{proof}

Thanks to this lemma we derive the following decomposition result.

\begin{theorem}\label{valuek}
Under the same assumptions of Theorem~\ref{thm-decomp}, every solution $u$ to \eqref{eq-1}-\eqref{bc-u} can be decomposed as
\begin{equation}\label{decomp2}
u(t,x)=\hat u^{R}(t,x) + k(t) \hat S(t,x)\,,
\end{equation}
where $k\in C^2([0,T])$ and $\hat u^R(t)\in H^2(\Omega\setminus\Gamma(t))$ for every $t\in [0,T]$. In particular the function $k$ does not depend on our choice of $\Phi$, but only on $\Gamma$ and $s$.
\end{theorem}

\begin{proof}
By combining the representation formula \eqref{decomp1} with Lemma~\ref{lemregxw}, we deduce the validity of the decomposition \eqref{decomp2} in $[t_0,t_1]$. Indeed we have 	
$$u(t)=(u^{R}(t)+k(t)\zeta(t)w(t)-k(t)(1-\zeta(t))\hat S(t)) + k(t)\hat S(t),\quad\text{in }[t_0,t_1],$$
being $w(t):=S(\Phi(t))-\hat S(t)$, and, by the previous result, $\hat u^R(t):=u^{R}(t)+k(t)\zeta(t)w(t) - k(t)(1-\zeta(t))\hat S(t)\in H^2(\Omega\setminus\Gamma(t))$.
	
We can now find a finite number of times $(t_i)_{i=1}^n$, with $0=t_0<t_1<\dots<t_{n-1}<t_n=T$ such that in every time interval $[t_{i-1},t_i]$  the solution $u$ to \eqref{eq-1}-\eqref{bc-u} is written as
\begin{equation}\notag
u(t,x)=\hat u^{R}_i(t,x) + k_i(t) \hat S(t,x)\,,
\end{equation}
with $k_i\in C^2([t_{i-1},t_i])$ and $\hat u^{R}_i(t,x)\in H^2(\Omega\setminus\Gamma(t))$. Define $k\colon[0,T]\to\mathbb R$ and $\hat u^{R}\colon [0,T] \to H^2(\Omega\setminus\Gamma)$ as $k(t):=k_i(t)$ and $\hat u^{R}:=\hat u_i^{R}$ in $[t_{i-1},t_i]$ for every $i=1,\dots,n$, respectively. The functions $k$ and $\hat u^R$ are well defined and do not depend on the particular choice of $(t_i)_{i=1}^n$. Indeed, if for some $t\in[0,T]$ we have

\begin{equation}\notag
u(t,x)=\hat u^{R}_1(t,x) + k_1(t) \hat S(t,x)=\hat u^{R}_2(t,x) + k_2(t) \hat S(t,x)\,,
\end{equation}
then we derive that
\begin{equation}\notag
\hat u^{R}_1(t)-\hat u^{R}_2(t)= (k_2(t)-k_1(t))\hat S(t)\,\quad\text{in }\Omega\setminus\Gamma(t)\,.
\end{equation}
Since the left--hand side belongs to $H^2(\Omega\setminus\Gamma(t))$ while $\hat S(t)$ is an element of $H^1(\Omega\setminus\Gamma(t))\setminus H^2(\Omega\setminus\Gamma(t))$, the only possibility to have such identity is that $k_1(t)=k_2(t)$ and $\hat u^{R}_1(t)=\hat u^{R}_2(t)$. Hence $k\in C^2([0,T])$ and $u$ satisfies the decomposition result \eqref{decomp2} in the whole $[0,T]$.
\end{proof}

We now want to recover the regularity in time for $\hat u^R$ and this is done in the following lemmas.

\begin{lemma}\label{lemregtw}
Under the same assumptions of Theorem~\ref{thm-decomp}, the function $\hat u^R$ introduced in \eqref{decomp2} is an element of $C^0([0,T];L^2(\Omega))$. Moreover, $\nabla \hat u^R\in C^0([0,T];L^2(\Omega;\mathbb R^2))$ and $\nabla^2\hat u^R\in C^0([0,T];L^2(\Omega;\mathbb R^{2\times 2}))$.
\end{lemma}

\begin{proof}
We start by proving that the function $w(t):=S(\Phi(t))-\hat S(t)$, already introduced in Lemma \ref{lemregxw}, satisfies the regularity properties of the thesis in $[t_0,t_1]$.
	
First, the function $S\circ \Phi$ belongs to $C^0([t_0,t_1];L^2(\Omega))$ in view of the fact that $S\in {C^\infty(\Omega^{(4)}\setminus\Gamma^{(4)})}\cap L^2(\Omega^{(4)})$ and that the diffeomorphism $\Phi$ is continuous in $[t_0,t_1]\times\overline \Omega$. We also claim that $\hat S=\tilde S\circ\tilde\Phi \in C^0([t_0,t_1]\times(\Omega\setminus\Gamma))\cap L^\infty((t_0,t_1)\times\Omega)$. Indeed let $(t^*,x^*)\in [t_0,t_1]\times(\Omega\setminus\Gamma)$ and let $(t_h,x_h)_{h\in\mathbb N}\subset[t_0,t_1]\times(\Omega\setminus\Gamma)$ be a sequence of points converging to $(t^*,x^*)$. Since $\tilde\Phi(t_h,x_h)\to\tilde\Phi(t^*,x^*)\in\tilde\Omega(t^*)\setminus\tilde\Gamma(t^*)$ as $h\to+\infty$, there exists $\bar h\in\mathbb N$ such that 
\begin{equation*}
\tilde S(t_h,\tilde\Phi(t_h,x_h))=\tilde S(t^*,\tilde\Phi(t_h,x_h))\quad\text{for every }h>\bar h\,.
\end{equation*}
This allows us to conclude that $\hat S(t_h,x_h)\to \hat S(t^*,x^*)$ as $h\to+\infty$, since the function $\tilde S(t^*)$ is continuous in $\tilde\Omega(t^*)\setminus\tilde\Gamma(t^*)$. Furthermore, $|\hat S(t,x)|\le M |\tilde\Phi(t,x)|^{\frac{1}{2}}$ for every $x\in\Omega\setminus\Gamma$ and $t\in[t_0,t_1]$ for a positive constant $M$, which gives us that $\hat S$ is uniformly bounded in $\Omega\setminus\Gamma$. We hence derive the claim, which implies that $\hat S \in C^0([t_0,t_1];L^2(\Omega))$ by the dominated convergence theorem.
	
Arguing as before, we can easily deduce that $\nabla (S\circ\Phi)\in C^0([t_0,t_1];L^2(\Omega;\mathbb R^2))$, while $\nabla \hat S=D\tilde\Phi^T(\nabla\tilde S)\circ\tilde\Phi\in C^0([t_0,t_1]\times(\Omega\setminus\Gamma);\mathbb R^2)$. By using also the estimate $|\nabla \tilde S(t,\tilde\Phi(t,x))|\le M|\tilde\Phi(t,x)|^{-\frac{1}{2}}$, which holds in $\Omega\setminus\Gamma$ for every $t\in[t_0,t_1]$, and the dominated converge theorem, we conclude that $\nabla \hat S\in C^0([t_0,t_1];L^2(\Omega;\mathbb R^2))$.
	
Finally, also the function $\nabla^2w$ is continuous in $[t_0,t_1]\times(\Omega\setminus\Gamma)$. Let us now fix $t^*\in[t_0,t_1]$ and let $(t_h)_{h\in\mathbb N}$ be a sequence of points in $[t_0,t_1]$ such that $t_h\to t^*$ as $h\to+\infty$. Thanks to the estimate \eqref{nabla2w}, we can find $\bar h\in\mathbb N$ and $\varepsilon>0$ such that
\begin{equation*}
|\nabla^2 w(t_h,x)|\le C|x-r(t_h)|^{-\frac{1}{2}}\quad\text{for every }x\in B_\varepsilon(r(t_h))\setminus\Gamma\text{ and }h>\bar h\,,
\end{equation*}
with $C$ independent of $h$. Here we have used the fact that the constant in \eqref{nabla2w} can be chosen uniform in time. Furthermore, the functions $\nabla^2 w(t_h)$ are uniformly bounded with respect to $h$ outside the ball $B_\varepsilon(r(t_h))$. Hence, by applying the generalized dominated convergence theorem, we deduce that $\nabla^2 w(t_h)\to \nabla^2 w(t^*)$ in $L^2(\Omega;\mathbb R^{2\times 2})$, which implies that $\nabla w^2\in C^0([t_0,t_1];L^2(\Omega;\mathbb R^{2\times 2}))$.
	
Combining the regularity of $w$ with the definition of $\hat u^R$, it is easy to see that $\hat u^R$ satisfies the thesis in $[t_0,t_1]$, and consequently in the whole $[0,T]$ by the arbitrariness of $[t_0,t_1]\subset[0,T]$.
\end{proof}

\begin{lemma}\label{lemregtw2}
Under the assumptions of Theorem~\ref{thm-decomp}, the function $\hat u^R$ introduced in \eqref{decomp2} is an element of $C^2([0,T];L^2(\Omega))$, moreover $\nabla\hat u^R\in C^1([0,T];L^2(\Omega;\mathbb R^2))$.
\end{lemma}

\begin{proof}
As before, it is enough to prove the validity of the thesis for the difference function $w(t):=S(\Phi(t))-\hat S(t)$, in the time interval $[t_0,t_1]$.
	
For every $x\in\Omega\setminus\Gamma$ the function $t\mapsto w(t,x)$ is differentiable in $[t_0,t_1]$ and 
\begin{equation*}
\dot w(t,x)=\frac{\de}{\de t}w(t,x)=\nabla S(\Phi(t,x))\cdot\dot\Phi(t,x)-\nabla \tilde S(t,\tilde \Phi(t,x))\cdot\dot{\tilde\Phi}(t,x)\,.
\end{equation*}
Indeed, fixed $(t^*,x^*)\in[t_0,t_1]\times(\Omega\setminus\Gamma)$, we can find $\bar h>0$ such that for every $|h|\le \bar h$
\begin{equation*}
\frac{\tilde S(t^*+h,\tilde\Phi(t^*+h,x^*))-\tilde S(t^*,\tilde\Phi(t^*,x^*))}{h}=\frac{\tilde S(t^*,\tilde\Phi(t^*+h,x^*))-\tilde S(t^*,\tilde\Phi(t^*,x^*))}{h}\,,
\end{equation*}
thanks to the fact that $\tilde\Phi(t^*+h,x^*)\to \tilde\Phi(t^*,x^*)\in \tilde\Omega(t^*)\setminus\tilde\Gamma(t^*)$ for every $x^*\in\Omega\setminus\Gamma$ as $h\to 0$. In particular $[\tilde S(t^*+h,\tilde\Phi(t^*+h,x^*))-\tilde S(t^*,\tilde\Phi(t^*,x^*))]/h\to \nabla \tilde S(t^*,\tilde \Phi(t^*,x^*))\cdot\dot{\tilde\Phi}(t^*,x^*)$, since $\tilde S(t^*)\in C^\infty(\tilde\Omega(t^*)\setminus\tilde\Gamma(t^*))$. Hence for every $(t,x)\in[t_0,t_1]\times(\Omega\setminus\Gamma)$ and $h\in\mathbb R$ such that $t+h\in[t_0,t_1]$ we may write
\begin{equation*}
\frac{w(t+h,x)-w(t,x)}{h}=\frac{1}{h}\int_t^{t+h}\dot w(\tau,x)\,\de\tau\,.
\end{equation*}
Arguing as in the proof of the previous lemma we deduce that $\dot w\in C^0([t_0,t_1];L^2(\Omega))$. Therefore we obtain that as $h\to 0$
\begin{equation*}
\frac{1}{h}\int_t^{t+h}\dot w(\tau)\,\de \tau\to \dot w(t)\quad\text{in }L^2(\Omega)\text{ for every }t\in[t_0,t_1]\,, 
\end{equation*}
and consequently $\frac{w(t+h)-w(t)}{h}\to \dot w(t)$ in $L^2(\Omega)$.
	
Similarly, for every $x\in\Omega\setminus\Gamma$ the map $t\mapsto\dot w(t,x)$ is differentiable in $[t_0,t_1]$ with derivative
\begin{align*}
\ddot w(t,x)=\frac{\de}{\de t}\dot w(t,x)&=\nabla S(\Phi(t,x))\cdot\ddot\Phi(t,x)-\nabla \tilde S(t,\tilde\Phi(t,x))\cdot\ddot{\tilde\Phi}(t,x)\\
&+\nabla^2 S(\Phi(t,x))\cdot[\dot\Phi(t,x)\otimes\dot\Phi(t,x)-\dot{\tilde\Phi}(t,x)\otimes\dot{\tilde\Phi}(t,x)]\\
&+[\nabla^2 S(\Phi(t,x))-\nabla^2 \tilde S(t,\tilde\Phi(t,x))]\dot{\tilde\Phi}(t,x)\otimes\dot{\tilde\Phi}(t,x)\,.
\end{align*}
Notice that we may find $\varepsilon>0$ so small that $|\dot\Phi(t,x)-\dot{\tilde\Phi}(t,x)|\le C|x-r(t)|$ in $B_\varepsilon(r(t))$ for every $t\in[t_0,t_1]$ and for a positive constant $C$. Therefore, proceeding as in the proof of Lemma~\ref{lemregxw}, we obtain that $\ddot w(t)\in L^2(\Omega)$ for every $t\in[t_0,t_1]$ with
\begin{equation*}
\left|\ddot w(t,x)\right|\le C|x-r(t)|^{-\frac{1}{2}}\quad\text{for every }x\in B_\varepsilon(r(t))\setminus\Gamma(t)\,.
\end{equation*}
In particular, arguing as in Lemma~\ref{lemregtw}, this uniform estimate implies that $\ddot w\in C^0([t_0,t_1];L^2(\Omega))$. We can hence repeat the same procedure adopted before for $\dot w$ to conclude that as $h\to 0$
\begin{equation*}
\frac{\dot w(t+h)-\dot w(t)}{h}\to\ddot w(t)\quad\text{in }L^2(\Omega)\text{ for every }t\in[t_0,t_1]\,,
\end{equation*}
which gives that $w\in C^2([t_0,t_1];L^2(\Omega))$.
	
Finally, also the function $t\mapsto\nabla w(t,x)$ is differentiable in $[t_0,t_1]$ for every $x\in\Omega\setminus\Gamma$ and
\begin{align*}
\nabla \dot w(t,x)=\frac{\de}{\de t}\nabla w(t,x)&=D \dot \Phi^T(t,x)\nabla S(\Phi(t,x))-D \dot{\tilde\Phi}^T(t,x)\nabla \tilde S(t,\tilde\Phi(t,x))\\
&+[D\Phi^T(t,x)-D\tilde\Phi^T(t,x)]\nabla^2 S(\Phi(t,x))\dot\Phi(t,x)\\
&+D\tilde\Phi^T(t,x)\nabla^2 S(\Phi(t,x))[\dot\Phi(t,x)-\dot{\tilde\Phi}(t,x)]\\
&+D\tilde\Phi^T(t,x)[\nabla^2 S(\Phi(t,x))-\nabla^2 \tilde S(t,\tilde\Phi(t,x))]\dot{\tilde\Phi}(t,x)\,.
\end{align*}
Moreover there exists $\varepsilon>0$ so small that for every $t\in[t_0,t_1]$ 
\begin{equation*}
\left|\nabla \dot w(t,x)\right|\le C|x-r(t)|^{-\frac{1}{2}}\quad\text{for every }x\in B_\varepsilon(r(t))\setminus\Gamma(t)\,,
\end{equation*}
which implies the continuity of the map $t\mapsto \nabla\dot w(t)$ from $[t_0,t_1]$ to $L^2(\Omega;\mathbb R^2)$. Therefore we get that as $h\to 0$
\begin{equation*}
\frac{\nabla w(t+h)-\nabla w(t)}{h}\to\nabla \dot w(t)\quad\text{in }L^2(\Omega;\mathbb R^2)\text{ for every }t\in[t_0,t_1]\,,
\end{equation*}
and in particular $\nabla w\in C^1([t_0,t_1];L^2(\Omega;\mathbb R^2))$.
\end{proof}

\begin{remark}\label{gencase}
When $A\neq I$ all the previous result are still true if we define
\begin{equation}\label{sing2}
\hat S(t,x) :=Im\left(\sqrt{\frac{A^{-1}(r(t))\left(x-r(t)\right)\cdot\gamma'(s(t))}{c_{A,\gamma'}(t)\sqrt{1-|c_{A,\gamma'}(t)|^2|\dot s(t)|^2}}+i\frac{\left(x-r(t)\right)\cdot n(s(t))}{c_{A,n}(t)}}\right)\,,
\end{equation}
where {$c_{A,\gamma'}(t):= |A^{-1/2}(r(t))\gamma'(s(t))|$, $c_{A,n}(t):= |A^{1/2}(r(t))n(s(t))|$, with $A^{1/2}$ and $A^{-1/2}$ the square root matrices of $A$ and $A^{-1}$}, respectively, and where $\hat S(t)$ is given by the unique continuous determination of the complex square function such that in $x=r(t)+\sqrt{1/|c_{A,\gamma'}(t)|^2-|\dot s(t)|^2}\gamma'(s(t))$ takes the value 1 and its discontinuity set lies on $\Gamma(t)$. Indeed, by exploiting the following identities in $[t_0,t_1]$
\begin{align*}
&(\gamma^{(1)})'(s^{(1)}(t))=\frac{A^{-1/2}(r(t))\gamma'(s(t))}{|A^{-1/2}(r(t))\gamma'(s(t))|}\,,\quad n^{(1)}(s^{(1)}(t))=\frac{A^{1/2}(r(t))n(s(t))}{|A^{1/2}(r(t))n(s(t))|}\,\\
&\dot s^{(1)}(t)=|A^{-1/2}(r(t))\gamma'(s(t))|\dot s(t)\,,\quad D\chi(r(t))=A^{-1/2}(r(t))\,,
\end{align*}
{where $(\gamma^{(1)})'$ and $n^{(1)}$ are, respectively, the tangent and the normal unit vectors to the curve $\Gamma^{(1)}$ in the point $\gamma^{(1)}(s^{(1)}(t))$,} the function \eqref{sing2} can be rewritten as
\begin{equation*}
Im\left(\sqrt{\frac{D\chi(r(t))\left(x-r(t)\right)\cdot(\gamma^{(1)})'(s^{(1)}(t))}{\sqrt{1-|\dot s^{(1)}(t)|^2} }+iD\chi(r(t))\left(x-r(t)\right)\cdot n^{(1)}(s^{(1)}(t))}\right)\,.
\end{equation*}
In this case it is enough to set $\tilde \Phi(t,x):=L(t)R(t)D\chi(r(t))(x-r(t))$, where $L$ and $R$ are constructed starting from $\gamma^{(1)}$ and $s^{(1)}$, and we can proceed again as in Lemmas~\eqref{lemregxw}, \eqref{lemregtw}, and \eqref{lemregtw2}, thanks to the fact that for every $t\in[t_0,t_1]$ and $x\in B_\varepsilon(r(t))$
\begin{align*}
&|\Phi(t,x)-\tilde\Phi(t,x)|\le C|x-r(t)|^2\,,\ |D\Phi(t,x)-D\tilde\Phi(t,x)|\le C|x-r(t)|\,,\\
&|\dot\Phi(t,x)-\dot{\tilde\Phi}(t,x)|\le C|x-r(t)|\,.
\end{align*}
We hence obtain the decomposition result \eqref{decomp2} with singular part \eqref{sing2}. As a byproduct, arguing as in Theorem~\ref{valuek}, we derive that the values of $k$ do not depend on the particular construction of $\Phi$, but only on $A$, $\Gamma$, and $s$.
	
We point out that the condition $|\dot s(t)|^2<1/|c_{A,\gamma'}(t)|^2$, which we need in order to define $\hat S$, is implied by \eqref{ipotesi-s}. Indeed
\begin{equation*}
1=D\chi(r(t))A(r(t))D\chi^T(r(t))\gamma'(s(t))\cdot \gamma'(s(t))\ge c_0|A(r(t))^{-1/2}\gamma'(s(t))|^2=c_0|c_{A,\gamma'}(t)|^2.
\end{equation*}
\end{remark}

\section{The energy-dissipation balance}\label{sec-err}

In this section we derive formula \eqref{energia} for the energy 
$$
\mathcal E(t):= \frac{1}{2} \| \dot{u}(t)\|_{L^2(\Omega)}^2 + \frac{1}{2} \| \nabla u(t) \|_{L^2(\Omega;\mathbb R^2)}^2\,,
$$ 
associated to $u$, solution to \eqref{eq-1}-\eqref{bc-u} with initial conditions $u(0)=u^0$, $\dot{u}(0)=u^1$.

The computation is divided into three steps: first, in Proposition \ref{enbalrect} we consider straight cracks when $A$ is the identity matrix; then, in Theorem \ref{thmeb} we adapt the techniques to curved fractures; finally, in Remark  \ref{genenba} we generalize the former results to $A\neq I$.
To this aim, some preliminaries are in order: first, in Remark \ref{compder} we compute the partial derivatives of $u$ in a more convenient way, then in Lemmas \ref{fondlem1} and \ref{tracelem} we provide two key results, based on Geometric Measure Theory. Once this is done, we deduce formula \eqref{edb} in the time interval $[t_0,t_1]$ where the decomposition \eqref{decomp1} holds. 

For brevity of notation, in this section we consider $[t_0,t_1]=[0,1]$. All the results can be easily extended to the general case.
The global result in $[0,T]$ easily follows by iterating the procedure a finite number of steps, and using both the additivity of the integrals and the fact that $k$ depends only on $A$, $\Gamma$, and $s$ (see Theorem~\ref{valuek} and Remark~\ref{gencase}).

\begin{remark}\label{compder}
Let us focus our attention on a fracture which is straight in a neighborhood of the tip. Without loss of generality, we may fix the origin so that for every $t\in[0,1]$
$$
\Gamma(t) \setminus \Gamma(0) = \{ (x_1,x_2) \in \mathbb{R}^2 \ : \ 0 < x_1 \leq s(t)-s(0), \ x_2 = 0 \}\,.
$$
The diffeomorphisms $\chi$ and $\Lambda$ introduced in \S \ref{ssec-cov} can be both taken equal to the identity, so that, in a neighborhood of the origin,
the diffeomorphisms $\Phi(t)$ defined in \eqref{def-Phi} simply read
$$
\Phi(t,x)= \left ( \frac{x_1- (s(t)-s(0))}{\sqrt{1 - |\dot{s}(t)|}}, x_2\right)\,.
$$
Accordingly, the decomposition result in Theorem \ref{thm-decomp} states that the solution $u$ to the wave equation \eqref{eq-1}-\eqref{bc-u} can be decomposed as
\begin{equation}
\label{decform}
u(t,x) = u^R(t,x) + k(t) \zeta(t,x) \overline{S}(t,x)\,,
\end{equation}
where, for brevity, we have set $\overline{S}(t,x) \coloneqq S(\Phi(t,x))$.
We recall that $u^R \in C^2([0,1]; L^2(\Omega))$, $\nabla u \in C^1([0,1]; L^2(\Omega; \mathbb{R}^2))$, $\nabla^2 u \in C^0([0,1]; L^2(\Omega; \mathbb{R}^{2 \times 2}))$, $u^R(t) \in H^2(\Omega \setminus \Gamma(t))$ for every $t \in [0,1]$, $k \in C^2([0,1])$, $\zeta \in C^1([0,1] \times \Omega)$, and $S(x) = \frac{x_2}{\sqrt{2}\sqrt{|x| + x_1}}$.

Let us now compute the partial derivatives of $u$. Since
\begin{eqnarray*}
&& \nabla S(x) = \frac{1}{2\sqrt{2}|x|} \bigg( \frac{-x_2}{\sqrt{|x| + x_1}}\,, \sqrt{|x| + x_1} \bigg)\,,
\\
&& \partial^2_{11}S(x) = \frac{2x_1x_2 + x_2|x|}{4\sqrt{2}|x|^3\sqrt{|x| + x_1}}\,, \ \ \partial^2_{22}S(x)= -\frac{2x_1x_2 + x_2|x|}{4\sqrt{2} |x|^3\sqrt{|x| + x_1}}\,,
\\
&& \partial^2_{12} S(x) = \partial^2_{21} S(x) = \frac{\sqrt{|x| + x_1} ( |x| -2 x_1)}{ 4\sqrt{2} |x|^3}\,,
\end{eqnarray*}
we get
\begin{eqnarray}
&& \nabla u(t,x)= \nabla u^R(t,x) + k(t) \nabla \zeta(t,x)  \overline{S}(t,x) + k(t) \zeta(t,x) \nabla \overline{S}(t,x)\,,\label{expl1}
\\
&& \dot{u}(t,x) = \dot{u}^R(t,x) + \dot{k}(t)\zeta(t,x) \overline{S}(t,x) +  k(t)  \dot{\zeta}(t,x) \overline{S}(t,x) + k(t) \zeta(t,x) \dot{\overline{S}}(t,x)\,.\label{expl2}
\end{eqnarray}
We claim that 
$$
\dot{u}(t) \nabla u(t) - k^2(t) \zeta^2(t) \dot{\overline{S}}(t) \nabla \overline{S}(t) \in W^{1,1}(\Omega \setminus \Gamma(t);\mathbb R^2)\,,
$$
 for every $t \in [0,1]$. 

In fact $\nabla u^R(t,x)$, $\zeta(t,x) \overline{S}(t,x)$, $ \dot{u}^R(t,x)$, $\zeta(t,x) \overline{S}(t,x)$, and $k(t)  \dot{\zeta}(t,x) \overline{S}(t,x)$ are functions in $W^{1,2}(\Omega \setminus \Gamma(t))$ for every $t \in [0,1]$; by the Sobolev embeddings theorem we deduce that each of the previous functions belongs to $L^p(\Omega \setminus \Gamma(t))$ for every $p \geq 1$; using also the explicit form of $\overline{S}(t,x)$ and $\dot{\overline{S}}(t,x)$, one can also check that both of these functions are elements of $W^{1,4/3}(\Omega \setminus \Gamma(t))$. Having this in mind, we can easily conclude that the products of each term appearing in \eqref{expl1} with each term appearing in \eqref{expl2}, except $k^2(t) \zeta^2(t) \dot{\overline{S}}(t) \nabla \overline{S}(t)$, are functions in $W^{1,1}(\Omega \setminus \Gamma(t);\mathbb R^2)$ for every $t \in [0,1]$.  
\end{remark}

\begin{lemma}
\label{fondlem1}
Let $a,b \in \mathbb{R}$ with $a < 0$ and $b>0$ and define $H^+ := \{ (x_1,x_2) \in \mathbb{R}^2 \ : \ x_2 \geq 0 \}$ to be the upper half plane in $\mathbb{R}^2$. Let $g \colon H^+ \to \mathbb{R}$ be bounded, continuous at the origin, and call $\omega$ a modulus of continuity for $g$ at $x = 0$. Then
\begin{equation}
\label{fondlem}
\bigg|\frac{1}{\epsilon}\int_0^\epsilon\bigg(\int_a^b g(x_1,x_2) \frac{x_2}{x_1^2+x_2^2} \, \de x_1\bigg)\,\de x_2  - \pi g(0,0) \bigg|\leq 
\|g\|_{L^\infty(H^+)} \big(2\varepsilon^{1/2}|b-a|+ \theta(\epsilon)\big) + \pi \omega(\epsilon^{1/4})\,,
\end{equation}
where
\[
\theta(\epsilon) := \bigg|\pi - \int_0^1 \arctan{\bigg(\frac{b}{\epsilon x_2} \bigg)} - \arctan{\bigg(\frac{a}{\epsilon x_2} \bigg)} \, \de x_2\bigg|\,.
\]
In particular, for every $g \colon H^+ \to \mathbb{R} $ bounded and continuous at the origin, we have
\begin{equation}\notag
\lim_{\epsilon \to 0^+ }\frac{1}{\epsilon}\int_0^\epsilon\bigg(\int_a^b g(x_1,x_2) \frac{x_2}{x_1^2+x_2^2} \, \de x_1\bigg)\, \de x_2 = \pi g(0,0)\,.
\end{equation}
\end{lemma}
\begin{proof}
After a change of variable on the integral in \eqref{fondlem}, we can rewrite it as 
$$
\int_0^1\bigg(\int_a^b g(x_1,\epsilon x_2) \frac{\epsilon x_2}{x_1^2+(\epsilon x_2)^2} \, \de x_1\bigg)\,\de x_2\,. 
$$
Note that
$$
\int_a^b \frac{\epsilon x_2}{x_1^2+(\epsilon x_2)^2} \, \de x_1 = \int_a^b \partial_{1} \arctan\bigg(\frac{x_1}{\epsilon x_2}\bigg) \, \de x_1 = \arctan\bigg(\frac{b}{\epsilon x_2}\bigg) - \arctan\bigg(\frac{a}{\epsilon x_2}\bigg)\,,
$$
therefore
\begin{equation}\notag
\begin{split}
&\bigg| \int_0^1\bigg(\int_a^b g(x_1,\epsilon x_2) \frac{\epsilon x_2}{x_1^2+(\epsilon x_2)^2} \, \de x_1\bigg)\,\de x_2 - \pi g(0,0) \bigg|  \\
&\leq\bigg| \int_0^1\bigg(\int_a^b [g(x_1,\epsilon x_2) - g(0,0)] \frac{\epsilon x_2}{x_1^2+(\epsilon x_2)^2} \, \de x_1\bigg)\,\de x_2  \bigg| +  g(0,0) \theta(\epsilon)  \\
&\leq\bigg| \int_0^1\bigg(\int_{(a,b) \setminus(-\epsilon^{1/4}, \epsilon^{1/4}) } \hspace{-0.2cm}[g(x_1,\epsilon x_2) - g(0,0)] \frac{\epsilon x_2}{x_1^2+(\epsilon x_2)^2} \, \de x_1\bigg)\, \de x_2 \bigg| +  \pi \omega(\epsilon^{1/4}) + g(0,0) \theta(\epsilon)\,. 
\end{split}
\end{equation}
Using the estimate
\[
\sup_{x \in [(a,b) \setminus (- \epsilon^{1/4}, \epsilon^{1/4})] \times (0,1)} \frac{\epsilon x_2}{(x_1^2 + (\epsilon x_2)^2)} \leq \frac{\epsilon^{1/2}}{1 + \epsilon^{3/2}} \leq \epsilon^{1/2}\,,
\]
valid  for every $\epsilon \in (0,1)$, we can continue the above chain of inequalities with
\begin{equation}\notag
\begin{split}
&\leq \epsilon^{1/2}\int_0^1\bigg(\int_{(a,b) \setminus(-\epsilon^{1/4}, \epsilon^{1/4}) } |g(x_1,\epsilon x_2) - g(0,0)| \, \de x_1\bigg)\,\de x_2  + \pi \omega(\epsilon^{1/4}) + g(0,0) \theta(\epsilon)  \\
&\leq2\epsilon^{1/2} \|g\|_{L^\infty(H^+)} |b-a| + \pi \omega(\epsilon^{1/4}) + g(0,0) \theta(\epsilon)\,,
\end{split}
\end{equation}
which is \eqref{fondlem}, and the proof is concluded.
\end{proof}

\begin{lemma}
\label{tracelem}
{Let $\Omega \subset \mathbb{R}^2$, let $\gamma \colon [0,\ell] \to \Omega$ be a Lipschitz curve, and set $\Gamma := \{\gamma(\sigma) \in \Omega \ : \ \sigma \in [0,\ell] \}$}. For every $\epsilon>0$ define $\varphi_\epsilon(x) := \frac{\emph{dist}(x, \Gamma)}{\epsilon} \wedge 1$. Then for each $u \in W^{1,1}(\Omega \setminus \Gamma)$ and for each $v \colon \Omega \to \mathbb{R}$ bounded and such that
$$
\lim_{x \to \overline{x}} v(x) = v(\overline{x}) \text{ for every } \overline{x} \in \Gamma\,,
$$
we have
\begin{equation}\notag
\lim_{\epsilon \to 0^+} \int_{\emph{dist}^+(x, \Gamma) < \epsilon} u(x) v(x) |\nabla \varphi_\epsilon (x) | \, \de x = \int_{\Gamma} u^+(y) v(y) \, \de \mathcal{H}^1(y)\,,
\end{equation}
where $u^+$ is the trace on $\Gamma$ from above and
\begin{equation}\notag
\{ \emph{dist}^+(x, \Gamma) < \epsilon \} := \bigcup_{\sigma \in [0,\ell]} B_\epsilon (\gamma(\sigma)) \cap \{ x \in \Omega \ : \ x \cdot (\gamma'(\sigma))^{\perp} > 0 \}\,.
\end{equation}
Equivalently,
\begin{equation}\notag
\lim_{\epsilon \to 0^+} \int_{\emph{dist}^-(x, \Gamma) < \epsilon} u(x)v(x) |\nabla \varphi_\epsilon (x) | \, \de x = \int_{\Gamma} u^-(y)v(y) \, \de \mathcal{H}^1(y)\,,
\end{equation}
where $u^-$ is the trace on $\Gamma$ from below and
\begin{equation}\notag
\{ \emph{dist}^-(x, \Gamma) < \epsilon \} := \bigcup_{\sigma \in [0,\ell]} B_\epsilon (\gamma(\sigma)) \cap \{ x \in \Omega \ : \ x \cdot (\gamma'(\sigma))^{\perp} < 0 \}\,.
\end{equation}
\end{lemma}

\begin{proof}
It is enough to apply the coarea formula to the Lipschitz maps $\varphi_{\epsilon}$.
\end{proof}

\begin{remark}
In what follows we compute the energy balance in the case of homogeneous Neumann conditions on the whole $\partial \Omega$. However, the same proof applies with no changes to the case of Dirichlet boundary conditions. For example, to treat the homogeneous Dirichlet condition on $\partial_D \Omega\subseteq \partial \Omega$, it is enough to check that the time derivative of the solution $\dot{u}(t)$ has still zero trace on $\partial \Omega$, in such a way that it still remains an admissible test function. But this is simply because the incremental quotient in time $[u(t+h) - u(t)]/h$ converges to $\dot{u}(t)$ as $h \to 0$, strongly in $H^1$ in a sufficiently small neighborhood of $\partial_D \Omega$, so that $\dot{u}$ has still zero trace on the Dirichlet part of the boundary.

Analogously, if we prescribe a regular enough non-homogeneous Dirichlet boundary condition, we can rewrite the wave equation changing the forcing term $f$ appearing in its right-hand side, and turn the non-homogeneous Dirichlet condition into a homogeneous one. Also in this case, the computations follow unchanged.
\end{remark}

\begin{proposition}
\label{enbalrect}
Let $\Omega \subset \mathbb{R}^2$ be a Lipschitz regular domain, and let $\big(\Gamma(t)\big)_{t \in [0,1]}$ be a family of rectilinear cracks inside $\Omega$, of the form
\[
\Gamma(t) \setminus \Gamma(0) := \{ (x_1,x_2) \in \mathbb{R}^2 \ : \ 0 < x_1 \leq s(t)-s(0), \ x_2 = 0 \}\,,
\]
where $s \in C^2([0,1])$
and $\dot{s}(t) \geq 0$ for every $t \in [0,1]$.

Suppose that a function $u \colon [0,1] \times \Omega \to \mathbb{R}$ can be decomposed as in \eqref{decform} and satisfies the wave equation with homogeneous Neumann boundary conditions on the boundary and on the cracks:
\begin{equation}
\label{waveeq1}
\ddot{u}(t) - \Delta u(t) = f(t) \text{ in } \Omega \setminus \Gamma(t)\,,
\end{equation}
for a.e. $t \in [0,1]$, with initial conditions $u(0) = u_0$ and $\dot{u}(0) = u_1$. Then for every $t \in [0,1]$, $u$ satisfies the energy balance
\begin{equation}
\label{enebal}
\mathcal E(t) - \mathcal E(0) + \mathcal{H}^1(\Gamma(t) \setminus \Gamma(0)) 
=\int_0^t \langle f(
\tau), \dot{u}(\tau) \rangle_{L^2(\Omega)} \, \de \tau
\end{equation}
if and only if the stress intensity factor $k$ is constantly equal to $\frac{2}{\sqrt{\pi}}$ in the set $\{\dot{s}>0\}$.
\end{proposition}

\begin{proof}
By hypothesis the function $u$ can be decomposed as $u(t,x) = u^R(t,x) + k(t)  \zeta(t,x) \overline{S}(t,x)$, where $u^R(t) \in H^2(\Omega \setminus \Gamma(t))$, $\zeta(t)$ is a cut--off function supported in a neighborhood of the moving tip of $\Gamma(t)$, and
$$
\overline{S}(t,x)=S\bigg( \frac{x_1 - (s(t)-s(0))}{\sqrt{1 - | \dot{s}(t)|^2}}, x_2 \bigg)\,,
$$
where $S(x_1, x_2) = \frac{x_2}{\sqrt{2} \sqrt{|x| + x_1}}$.

Fix $\overline{t} \in [0,1]$. For every $\epsilon > 0$ define $\varphi_{\epsilon}(x) = \frac{\text{dist}(x, \Gamma(\overline{t}) \setminus \Gamma(0))}{\epsilon} \wedge 1$. Since $\varphi_\epsilon \dot{u}(t) \in H^{1}(\Omega  \setminus \Gamma(t))$, we can use it as test function in \eqref{waveeq1}, and we get
\begin{equation}
\label{wetested}
\begin{split}
\int_0^{\overline{t}} \langle \ddot{u}(t), \varphi_{\epsilon} \dot{u}(t) \rangle_{H^1(\Omega)} \, \de t& + \int_0^{\overline{t}} \langle \nabla u(t), \nabla \dot{u}(t) \varphi_{\epsilon} \rangle_{L^2(\Omega;\mathbb R^2)} \, \de t \\
&+ \int_0^{\overline{t}} \langle \nabla u(t), \nabla \varphi_\epsilon \dot{u}(t) \rangle_{L^2(\Omega;\mathbb R^2)} \, \de t = \int_0^{\overline{t}} \langle f(t), \dot{u}(t) \varphi_\epsilon \rangle_{L^2(\Omega)} \, \de t\,.
\end{split}
\end{equation}
Using integration by parts with the fact that $t \mapsto \| \dot{u}(t)\|^2_{L^2(\Omega, \varphi_\epsilon \de x)}$ is absolutely continuous, we obtain
\begin{equation}\notag
\begin{split}
\int_0^{\overline{t}} \langle &\ddot{u}(t), \varphi_{\epsilon} \dot{u}(t) \rangle_{H^1(\Omega)} \, \de t = \frac{1}{2}\int_0^{\overline{t}}  \frac{\de }{\de t} \| \dot{u}(t) \|^2_{L^2(\Omega,\varphi_\epsilon \de x)} \, \de t \\
&= \frac{1}{2} \| \dot{u}(\overline{t}) \|^2_{L^2(\Omega, \varphi_{\epsilon} \de x)} - \frac{1}{2}\| \dot{u}(0) \|^2_{L^2(\Omega, \varphi_{\epsilon} \de x)}\,, 
\end{split}
\end{equation}
and passing to the limit as $\epsilon \to 0^+$, by dominated convergence Theorem, we have
\begin{equation}\notag
\lim_{\epsilon \to 0^+} \int_0^{\overline{t}} \langle \ddot{u}(t), \varphi_{\epsilon} \dot{u}(t) \rangle_{H^1(\Omega)} \, \de t  = \frac{1}{2} \| \dot{u}(\overline{t}) \|^2_{L^2(\Omega)} - \frac{1}{2} \| \dot{u}(0) \|^2_{L^2(\Omega)}\,.
\end{equation}
Analogously, taking the limit as $\epsilon \to 0$ in the second term in the left-hand side and in the right-hand side of \eqref{wetested}, we have, respectively,
\begin{eqnarray*}
&& \lim_{\epsilon \to 0} \int_0^{\overline{t}} \langle \nabla u(t), \nabla \dot{u}(t) \varphi_{\epsilon} \rangle_{L^2(\Omega;\mathbb R^2)} \, \de t = \frac{1}{2} \| \nabla u(\overline{t}) \|^2_{L^2(\Omega;\mathbb R^2)} - \frac{1}{2} \| \nabla u(0) \|^2_{L^2(\Omega;\mathbb R^2)}\,,
\\
&& 
\lim_{\epsilon \to 0} \int_0^{\overline{t}} \langle f(t), \dot{u}(t) \varphi_\epsilon \rangle_{L^2(\Omega)} \, \de t = \int_0^{\overline{t}} \langle f(t), \dot{u}(t) \rangle_{L^2(\Omega)} \, \de t\,.
\end{eqnarray*}
The most delicate term is the third one in the left-hand side of \eqref{wetested}. First of all, we write the partial derivatives explicitly:
 \begin{eqnarray*}
&&\nabla [ k(t) \zeta(t,x) \overline{S}(t,x)] = k(t) \nabla \zeta(t,x)  \overline{S}(t,x) + k(t) \zeta(t,x) \nabla \overline{S}(t,x)\,,
\\
&&
\frac{\de}{\de t} [k(t) \zeta(t,x) \overline{S}(t,x)] =\dot{k}(t)\zeta(t,x) \overline{S}(t,x) +  k(t)  \dot{\zeta}(t,x) \overline{S}(t,x) + k(t) \zeta(t,x) \dot{\overline{S}}(t,x)\,.
\end{eqnarray*}
Moreover, if we set $\Phi_1(t,x) =  \frac{x_1 - s(t)}{\sqrt{1 - |\dot{s}(t)|^2}}$, we have
\[
\nabla \overline{S}(t,x) = \bigg( \frac{1}{\sqrt{1 - |\dot{s}(t)|^2}} {\partial_1} S( \Phi_1(t,x), x_2 ), {\partial_2} S( \Phi_1(t,x), x_2 )\bigg)
\]
and
\begin{equation}\notag
\begin{split}
\dot{\overline{S}}(t,x) &= \bigg[ \frac{-\dot{s}(t) (1 - |\dot{s}(t)|^2) + \dot{s}(t) \ddot{s}(t)(x_1 - (s(t)-s(0)))}{(1 - |\dot{s}(t)|^2)^{3/2}}  \bigg] {\partial_1} S( \Phi_1(t,x), x_2 ) \\
&= \dot{\Phi}_1(t,x)\sqrt{1-|\dot s(t)|^2}  {\partial_1}\bar S(t,x)\,.
\end{split}
\end{equation}
Thanks to Remark \ref{compder}, we know that the only contribution to the limit as $\epsilon \to 0$ is given by the following term:
\[
\int_0^{\overline{t}} k^2(t) \langle \zeta^2(t,x) \nabla \overline{S}(t,x), \nabla\varphi_{\epsilon}(x) \dot{\overline{S}}(t,x) \rangle_{L^2(\Omega;\mathbb R^2)} \, \de t\,.
\]
Therefore, we need to compute
\begin{equation}
\label{prinlim}
\lim_{\epsilon \to 0^+}\int_0^{\overline{t}}\bigg( \int_{\{ \text{dist}(x, \Gamma(t) \setminus \Gamma(0)) < \epsilon\}} k^2(t)  \zeta^2(t,x) \nabla \overline{S}(t,x) \cdot \nabla\varphi_{\epsilon}(x) \dot{\overline{S}}(t,x)\, \de x \bigg)  \, \de t\,.
\end{equation}
To this aim, we set $I_\epsilon(t) := \int_{\{ \text{dist}(x, \Gamma(t) \setminus \Gamma(0)) < \epsilon\}} k^2(t) \zeta^2(t,x) \nabla \overline{S}(t,x) \cdot \nabla\varphi_{\epsilon}(x) \dot{\overline{S}}(t,x)  \, \de x$ and we decompose $I_\epsilon$ as $I^+_{\epsilon} + I_{\epsilon}^-$, where $I^+_{\epsilon}$ is the integral $I_\epsilon$ restricted to the upper half plane $\{ x_2 > 0 \}$ and $I^-_\epsilon$ is the integral $I_\epsilon$ restricted to the lower half plane $\{ x_2 < 0 \}$.

Let us focus on $I_\epsilon^+(t)$.  \\
For brevity, we write $r(t) := (s(t)-s(0),0)$ for every $t \in [0,T]$. Then the gradient of $\varphi_\epsilon$ reads
\begin{equation}\notag
\nabla \varphi_{\epsilon} =
\begin{cases}
\frac{e_2}{\epsilon} &\text{ in } \{ x \in \mathbb{R}^2 \ : \  0 \leq x_1 \leq s(\overline{t})-s(0), \ 0 \leq x_2 < \epsilon \} \\
\frac{x}{\epsilon|x|} &\text{ in } \{ x \in \mathbb{R}^2 \ : \  x \in B_{\epsilon}(0), \ x_1 < 0, \ x_2 \geq 0 \}  \\ 
\frac{x - {r(\overline t)}}{\epsilon|x - {r(\overline t)}|} &\text{ in } \{ x \in \mathbb{R}^2 \ : \ x \in B_{\epsilon}({r(\overline t)}), \ x_1 > s(\overline{t})-s(0), \ x_2 \geq 0 \} \\
0 &\text{ otherwise on } \{ x_2 \geq 0 \}\,.
\end{cases}
\end{equation}

Thus we get
\begin{equation}
\label{seppie}
\begin{split}
I^+_\epsilon(t) & = \frac{1}{\epsilon} \int_{[0,s(\overline{t})-s(0)] \times (0, \epsilon]} \hspace{-0,5cm}k^2(t) \sqrt{1-|\dot s(t)|^2} \zeta^2(t,x) {\partial_2} \overline{S}(t,x) \dot{\Phi}_1(t,x) {\partial_1} \overline{S}(t,x) \, \de x \\
&+ \frac{1}{\epsilon} \int_{B_\epsilon(0) \cap \{x_1 < 0\} \times \{x_2 \geq 0 \}} \hspace{-1,15cm} k^2(t) \sqrt{1-|\dot s(t)|^2} \zeta^2(t,x) \bigg(\nabla \overline{S}(t,x) \cdot \frac{x}{|x|}\bigg) \dot{\Phi}_1(t,x) {\partial_1} \overline{S}(t,x) \, \de x \\
&+ \frac{1}{\epsilon} \int_{B_\epsilon({r(\overline t)}) \cap \{x_1 > s(\overline{t})-s(0)\} \times \{x_2 \geq 0 \}} \hspace{-2,5cm}k^2(t) \sqrt{1-|\dot s(t)|^2} \zeta^2(t,x) \bigg(\nabla \overline{S}(t,x) \cdot \frac{x- {r(\overline t)}}{|x - {r(\overline t)}|} \bigg) \dot{\Phi}_1(t,x) {\partial_1} \overline{S}(t,x) \, \de x\,.
\end{split}
\end{equation}
We notice that the last two terms in \eqref{seppie} have integrands which are bounded on the domains of integration, and so passing to the limit as $\epsilon$ goes to $0$ they do not give any contribution. Thus we only have to analyze the first term of \eqref{seppie}. Recalling that $\zeta(x,t)=\zeta\big(\Phi_1(t,x) ,x_2\big)$, $\overline{S}(t,x) = S\big(\Phi_1(t,x) ,x_2\big)$, $\Phi_1(t,x) =  \frac{x_1 - (s(t)-s(0))}{\sqrt{1 - |\dot{s}(t)|^2}}$,
and making the change of variable $x'_1\sqrt{1 - |\dot{s}(t)|^2} = x_1 - (s(t)-s(0))$, we rewrite the first term of \eqref{seppie} as
\begin{equation}
\label{finte1}
\begin{split}
&-\frac{k^2(t)\dot{s}(t) }{\epsilon} \bigg(\int_0^\epsilon \int_{a_t}^{b_t} \zeta^2(x_1, x_2) {\partial_1} S(x_1,x_2) {\partial_2} S(x_1,x_2) \, \de x_1\de x_2\bigg)   \\ +
&\frac{k^2(t)\dot{s}(t) \ddot{s}(t) }{\epsilon \sqrt{1 - |\dot{s}(t)|^2}} \bigg(\int_0^\epsilon \int_{a_t}^{b_t} x_1 \zeta^2(x_1, x_2)   {\partial_1} S(x_1,x_2) {\partial_2} S(x_1,x_2) \, \de x_1\de x_2\bigg)\,,
\end{split}
\end{equation}
where the interval $(a_t, b_t)$ denotes the segment  $\frac{\{ 0 < x_1 <s(\overline{t})-s(0)\} - (s(t)-s(0))}{\sqrt{1- |\dot{s}(t)|^2}}$.

Notice that
\begin{equation}\notag
\begin{split}
-\frac{k^2(t)\dot{s}(t) }{\epsilon} \bigg(\int_0^\epsilon & \int_{a_t}^{b_t} \zeta^2(x_1, x_2) {\partial_1} S(x_1,x_2) {\partial_2} S(x_1,x_2) \, \de x_1\de x_2\bigg) \\ = 
&\frac{k^2(t)\dot{s}(t) }{\epsilon} \bigg(\int_0^\epsilon \int_{a_t}^{b_t} \zeta^2(x_1, x_2)  \frac{x_2} {8|x|^2} \, \de x_1\de x_2\bigg)
\end{split}
\end{equation}
and that the function $(x_1,x_2) \mapsto \zeta^2(x_1, x_2) $ is bounded and continuous in $(0,0)$, therefore we are in a position to apply Lemma \ref{fondlem1}, which gives, in the limit as $\epsilon \to 0^+$,
\begin{equation}\notag
\begin{split}
\lim_{\epsilon \to 0^+}\frac{k^2(t)\dot{s}(t) }{\epsilon} \bigg(\int_0^\epsilon \int_{a_t}^{b_t} \zeta^2(x_1, x_2)  \frac{x_2} {8|x|^2} \, \de x_1\de x_2\bigg) 
=\frac{\pi}{8}k^2(t) \dot{s}(t) \zeta^2(0,0)\,.
\end{split}
\end{equation}

Arguing in the very same way, we can show that the limit as $\epsilon \to 0^+$ of the second term of \eqref{finte1}, thanks to the presence of $x_1$, is zero. This means that the limit of $I_\epsilon^+(t)$ is
\[
\lim_{\epsilon \to 0^+} I^+_\epsilon(t) =  \frac{\pi}{8} \dot{s}(t)k^2(t)\,,
\]
and, similarly,
\[
\lim_{\epsilon \to 0^+} I_\epsilon^-(t) = \frac{\pi}{8} \dot{s}(t)k^2(t)\,.
\]
All in all,
\[
\lim_{\epsilon \to 0^+} I_\epsilon(t) = \lim_{\epsilon \to 0^+} [I^+_\epsilon(t) + I^-_\epsilon(t)]=\frac{\pi}{4} k^2(t)\dot{s}(t)\,.
\]
Thanks to the estimate in \eqref{fondlem}, we infer that the family of functions $(I^+_\epsilon(t))_{\epsilon > 0}$ are dominated on $[0,1]$ by a bounded function, and the same holds for $(I^-_\epsilon(t))_{\epsilon > 0}$; by the dominated convergence Theorem, we can pass the limit in \eqref{prinlim} inside the integral in time, and we can write
\[
\lim_{\epsilon \to 0^+}\int_0^{\overline{t}}I_\epsilon(t) \, \de t = \int_0^{\overline{t}} \frac{\pi}{4} k^2(t) \dot{s}(t) \, \de t\,.
\]
So we deduce that the energy balance in \eqref{enebal} holds for every $\overline{t} \in [0,1]$ if and only if the stress intensity factor $k(t)$ is equal to $\frac{2}{\sqrt{\pi}}$ whenever $\dot{s}(t)>0$.
\end{proof}

\begin{remark} We underline that our approach is different to that of Dal Maso, Larsen, and Toader \cite[\S 4]{DMLT}: in order to derive the energy balance associated to a horizontal crack opening with constant velocity $c$, they prove that the kinetic+elastic energy of $u(t)$ is constant in the moving ellipse $E_r(t)=\{(x_1,x_2)\in \mathbb R^2\ :\ (x_1-ct)^2/(1-c^2) + x_2^2 \leq r^2\}$ centered at the crack tip $(ct,0)$, for some small $r>0$, and they make the explicit computation of the energy in $\mathbb R^2 \setminus E_r(t)$.
\end{remark}

We now generalize the previous result to non straight cracks. 

\begin{theorem}\label{thmeb}
Let $\Omega \subset \mathbb{R}^2$ be a Lipschitz regular domain, and let $\big(\Gamma(t)\big)_{t \in [0,1]}$ be a family of growing cracks inside $\Omega$. Assume that there exists a bi-Lipschitz map $\Lambda \colon \Omega \to \Omega$ with the following properties:
\begin{enumerate}
\item $\Lambda(\Gamma(t) \setminus \Gamma(0)) = \{ (x_1,x_2) \in \mathbb{R}^2 \ : \ 0< x_1 \leq s(t)-s(0), \, x_2=0  \}$, where $s \in C^2([0,1])$
and $\dot{s}(t) \geq 0$ for every $t \in [0,1]$,
\item
$\mathcal{H}^1\big(\Lambda(\Gamma(t) \setminus \Gamma(0))\big) = \mathcal{H}^1\big(\Gamma(t) \setminus \Gamma(0)\big) \text{ for every } t \in [0,1]$,
\item
$\lim_{x \to \overline{x}}\nabla \Lambda(x) = \nabla \Lambda(\overline{x}) \in \text{SO}(2)^+, \text{ for every } \overline{x} \in \overline{\Gamma(1) \setminus \Gamma(0)}$.
\end{enumerate}
Suppose that a function $u \colon [0,1] \times \Omega \to \mathbb{R}$ can be decomposed as in \eqref{decform}
and satisfies the wave equation with homogeneous Neumann boundary conditions on the boundary and on the cracks:
\begin{equation}
\label{waveeq2}
\ddot{u}(t) - \Delta u(t) = f(t) \text{ in } \Omega \setminus \Gamma(t)\,,
\end{equation}
for a.e. $t \in [0,1]$, with initial conditions $u(0) = u_0$ and $\dot{u}(0) = u^1$. Then for every $t \in [0,1]$, $u$ satisfies the following energy balance
\begin{equation}
\label{enebal2}
\mathcal E(t) - \mathcal E(0) +  \mathcal{H}^1(\Gamma(t) \setminus \Gamma(0)  ) 
= \int_0^t \langle f(\tau), \dot{u}(\tau) \rangle_{L^2(\Omega)} \, \de \tau
\end{equation}
if and only if the stress intensity factor $k$ is constantly equal to $\frac{2}{\sqrt{\pi}}$ in the set $\{\dot{s}>0\}$.
\end{theorem}

\begin{proof}
In view of \eqref{decform}, we have $u(t,x) = u^R(t,x) + k(t)  \zeta(t,\Lambda(x)) \overline{S}(t,\Lambda(x))$, with $u^R(t) \in H^2(\Omega \setminus \Gamma(t))$, $\zeta(t,\Lambda(\cdot))$ a cut--off function supported in a neighborhood of the moving tip of $\Gamma(t)$, and
$$
\overline{S}(t,\Lambda(x)) = S\bigg( \frac{\Lambda_1(x) - (s(t)-s(0))}{\sqrt{1 - | \dot{s}(t)|^2}}, \Lambda_2(x) \bigg)\,,
$$ 
where $S(x_1, x_2) = \frac{x_2}{\sqrt{2} \sqrt{|x| + x_1}}$.

As in the proof of Proposition \ref{enbalrect}, we fix $\overline{t} \in [0,1]$ and, for every $\epsilon > 0$, we define $\varphi_{\epsilon}(x) = \frac{\text{dist}(x, \Gamma(\overline{t}) \setminus \Gamma(0))}{\epsilon} \wedge 1$. Since $\varphi_\epsilon \dot{u}(t) \in H^{1}(\Omega  \setminus \Gamma(t))$, we can use it as test function in \eqref{waveeq2}, and we get
\begin{equation}
\label{wetested10}
\begin{split}
\int_0^{\overline{t}} \langle \ddot{u}(t), \varphi_{\epsilon} \dot{u}(t) \rangle_{H^1(\Omega)} \, \de t &+ \int_0^{\overline{t}} \langle \nabla u(t), \nabla \dot{u}(t) \varphi_{\epsilon} \rangle_{L^2(\Omega;\mathbb R^2)} \, \de t \\
&+ \int_0^{\overline{t}} \langle \nabla u(t), \nabla \varphi_\epsilon \dot{u}(t) \rangle_{L^2(\Omega;\mathbb R^2)} \, \de t 
= \int_0^{\overline{t}} \langle f(t), \dot{u}(t) \varphi_\epsilon \rangle_{L^2(\Omega)} \, \de t\,.
\end{split}
\end{equation}
Integrating by parts, we easily obtain
\begin{eqnarray}
\label{easy1}&&\lim_{\epsilon \to 0^+} \int_0^{\overline{t}} \langle \ddot{u}(t), \varphi_{\epsilon} \dot{u}(t) \rangle_{H^1(\Omega)} \, \de t  = \frac{1}{2} \| \dot{u}(\overline{t}) \|^2_{L^2(\Omega)} - \frac{1}{2} \| \dot{u}(0) \|^2_{L^2(\Omega)}\,,
\\
\label{easy2}&&\lim_{\epsilon \to 0^+} \int_0^{\overline{t}} \langle \nabla u(t), \nabla \dot{u}(t) \varphi_{\epsilon} \rangle_{L^2(\Omega;\mathbb R^2)} \, \de t = \frac{1}{2} \| \nabla u(\overline{t}) \|^2_{L^2(\Omega;\mathbb R^2)} - \frac{1}{2} \| \nabla u(0) \|^2_{L^2(\Omega;\mathbb R^2)}\,,
\\
\label{easy3}&&\lim_{\epsilon \to 0^+} \int_0^{\overline{t}} \langle f(t), \dot{u}(t) \varphi_\epsilon \rangle_{L^2(\Omega)} \, \de t = \int_0^{\overline{t}} \langle f(t), \dot{u}(t) \rangle_{L^2(\Omega)} \, \de t\,.
\end{eqnarray}
The asymptotics as $\epsilon \to 0$ of the third term in the left-hand side of \eqref{wetested10} is more delicate to handle. To simplify the notation, we set
$$
\overline{\zeta}(t,x) := \zeta(t,\Lambda(x))\text{ and } \overline{\varphi}_\epsilon(x) := \varphi_\epsilon(\Lambda^{-1}(x))\,.
$$
Using Lemma \ref{tracelem} and Remark \ref{compder}, as in the proof of the previous proposition in the rectilinear case, we have that the only contribution to the limit as $\epsilon \to 0$ is given by the term
\begin{equation}
\label{samecomp}
\begin{split}
&\int_{\Omega}k^2(t){\overline{\zeta}}^2(t,x) \big[ \nabla \overline{S}(t, \Lambda(x)) \cdot \nabla \varphi_\epsilon(x)\big] \dot{\overline{S}}(t,\Lambda(x))\, \de x \\ 
&= \int_{\Omega}k^2(t) {\alpha(t)}\big[ D \Lambda^T(x) \nabla \overline{S}(t,\Lambda(x)) \cdot \nabla \varphi_\epsilon(x) \big]\overline{\zeta}^2(t,x)  \dot{\Phi}_1(t,\Lambda(x)) {\partial_1}\overline{S}(t, \Lambda(x)) \, \de x \\ 
&=\int_{\Omega}k^2(t) {\alpha(t)} \big[ D \Lambda^T (\Lambda^{-1}(x)) \nabla \overline{S}(t,x) \cdot  \nabla \varphi_\epsilon(\Lambda^{-1}(x))\big] \zeta^2(t,x) \dot{\Phi}_1(t,x) {\partial_1}\overline{S}(t, x)\, |J\Lambda^{-1}(x)|\, \de x\\ 
&= \int_{\Omega} k^2(t) {\alpha(t)}\big[ \nabla \overline{S}(t,x) \cdot B(\Lambda^{-1}(x)) \nabla  \overline{\varphi}_\epsilon(x) \big] \zeta^2(t,x) \dot{\Phi}_1(t,x) {\partial_1}\overline{S}(t, x)\, |J\Lambda^{-1}(x)|\, \de x\,, 
\end{split}
\end{equation}
where $\Phi_1(t,x) := \frac{x_1 - (s(t)-s(0))}{\sqrt{1 - |\dot{s}(t)|^2}}$, $B(x) := D\Lambda(x) D\Lambda^T(x)$, and {$\alpha(t):=\sqrt{1-|\dot s(t)|^2}$}. In the last equality we used the coarea formula applied with the Lipschitz change of variables $\Lambda^{-1}$. 

Thanks to our construction of $\Lambda$, for any $x$ belonging to a suitable small neighborhood of $\{ \Lambda (\Gamma(1)) \}$ we have 
$$
B(\Lambda^{-1}(x)) =  
\begin{pmatrix}
b_{11}(x) & 0 \\
0 & 1
\end{pmatrix}\,,
$$
where $b_{11} \colon \mathbb{R}^2 \to \mathbb{R}$ is a continuous function such that $b_{11}(x_1, 0) = 1$. The last term in \eqref{samecomp} can be split as 
\begin{equation}\notag
\begin{split}
\int_{\Omega}k^2(t) \alpha(t) b_{11}(x){\partial_1}\overline{\varphi}_\epsilon(x)  \zeta^2(t,x) \dot{\Phi}_1(t,x)& [{\partial_1}\overline{S}(t, x)]^2 \, |J\Lambda^{-1}(x)|\, \de x  \\ +
\int_{\Omega}k^2(t) \alpha(t){\partial_2} \overline{S}(t,x) {\partial_2}\overline{\varphi}_\epsilon(x)  \zeta^2(t,x)& \dot{\Phi}_1(t,x) {\partial_1}\overline{S}(t, x)  \, |J\Lambda^{-1}(x)|\, \de x\,.
\end{split}
\end{equation}
By construction of $\Lambda$, each line parallel to $\{ x_2 = 0\}$ is mapped by $\Lambda^{-1}$ into a level set of $\varphi_{\epsilon}$; more precisely $\varphi_{\epsilon}(\Lambda^{-1}(\{ x_2 = s \})) = \frac{s}{\epsilon} \wedge 1$, and this means that on the set of points $\{ \text{dist}(x, \Lambda(\Gamma(1))) \leq \epsilon \}$, we have
\begin{equation}\notag
\nabla \overline{\varphi}_{\epsilon}(x) =
\begin{cases}
\frac{e_2}{\epsilon} &\text{ in } \{ x \in \mathbb{R}^2 \ : \  0 \leq x_1 \leq s(\overline{t})-s(0), \ 0 \leq x_2 < \epsilon \} \\
\frac{x}{\epsilon|x|} &\text{ in } \{ x \in \mathbb{R}^2 \ : \  x \in B_{\epsilon}(0), \ x_1 < 0, \ x_2 \geq 0 \}  \\ 
\frac{x - {r(\overline t)}}{\epsilon|x - {r(\overline t)}|} &\text{ in } \{ x \in \mathbb{R}^2 \ : \ x \in B_{\epsilon}({r(\overline t)}), \ x_1 > s(\overline{t})-s(0), \ x_2 \geq 0 \} \\
0 &\text{ otherwise on } \{ x_2 \geq 0 \}\,, 
\end{cases}
\end{equation}
where, for brevity, we have set ${r(t)} := (s(t)-s(0),0)$ for every $t \in [0,1]$.

Since $\Lambda$ is a bi-Lipschitz map, $|J \Lambda^{-1}|$ is bounded, thus by hypothesis (3) we have
$$
\lim_{x \to (s(t)-s(0),0)} |J \Lambda^{-1}(x)| = 1\,,
$$
for every $t \in [0,1]$.

Moreover, in view of assumption (3), we have that $|J \Lambda^{-1}|$ is continuous on the compact set $\overline{\Gamma(1) \setminus \Gamma(0)}$, hence uniformly continuous; therefore, proceeding exactly as in the proof of Proposition \ref{enbalrect}, 
we can write
\begin{equation}
\label{convtfix}
\lim_{\epsilon \to 0^+} \int_{\Omega}k^2(t) {\alpha( t)}{\partial_2} \overline{S}(t,x) {\partial_2}\overline{\varphi}_\epsilon(x)  \zeta^2(t,x) \dot{\Phi}_1(t,x) {\partial_1}\overline{S}(t, x)  \, |J\Lambda^{-1}(x)|\, \de x \\ = \frac{\pi}{4} k^2(t) \dot{s}(t)\,. 
\end{equation} 
Again by hypothesis (3), we can apply estimate \eqref{fondlem} and deduce that the sequence of integrands in \eqref{convtfix} is dominated in $t$, so that we can apply the Dominated Convergence Theorem to deduce 
\begin{equation}
\label{convtnfix}
\begin{split}
\lim_{\epsilon \to 0^+} \int_0^{\overline{t}}\bigg( \int_{\Omega}k^2(t) {\alpha(t)}{\partial_2} \overline{S}(t,x) {\partial_2}\overline{\varphi}_\epsilon(x)  \zeta^2(t,x) \dot{\Phi}_1(t,x) {\partial_1}\overline{S}(t, x)  \, |J\Lambda^{-1}(x)|\, \de x\bigg) \, \de t \\
=\int_0^{\overline{t}} \frac{\pi}{4}k^2(t)\dot{s}(t) \, \de t\,.
\end{split}
\end{equation} 
By combining \eqref{wetested10} with \eqref{easy1}-\eqref{easy3} and \ref{convtnfix}, we infer that
\begin{equation}\label{lunedi}
\mathcal{E}(\overline{t}) - \mathcal{E}(0) + \frac{\pi}{4}\int_0^{\overline{t}} k^2(t)\dot{s}(t) \, \de t = \int_0^{\overline{t}} \langle f(t), \dot{u}(t) \rangle_{L^2(\Omega)} \, \de t\,.
\end{equation}
Hence, the energy-dissipation balance \eqref{enebal2} is satisfied if and only if 
\[
\int_0^{\overline{t}} \frac{\pi}{4}k^2(t)\dot{s}(t) \, \de t = s(\overline{t}) = \mathcal{H}^1\big(\Lambda(\Gamma(\overline{t}) \setminus \Gamma(0))\big) = \mathcal{H}^1(\Gamma(\overline{t})\setminus \Gamma(0)) \text{ for every } \overline{t} \in [0,1]\,,
\]
which is true if and only if $k(t)$ is equal to $\frac{2}{\sqrt{\pi}}$ whenever $\dot{s}(t)>0$. This concludes the proof.
\end{proof}

\begin{remark}
Our approach is constructive and allows us to show the existence of pairs $(\Gamma(t), u(t))$ satisfying the energy-dissipation balance \eqref{enebal2}. Under the standing assumptions on $\Gamma(t)$, it is enough to take $f$ associated to $2/\sqrt{\pi} \xi(\Phi(t,x))S(\Phi(t,x))$ (which of course is $u(t)$), where $\xi$ is a suitable cut--off function supported in a small neighborhood of the origin. In order to ensure the homogeneous Neumann condition on the fracture, we choose $\xi$ satisfying $\partial_2\xi(y_1,0)=0$ for every $y_1\in\mathbb R$. This can be achieved, e.g., by taking $\xi(y_1,y_2)= \varphi(y_1)\varphi(y_2)$, where $\varphi\in C^\infty_c(\mathbb R)$ has compact support contained in $(-\varepsilon,\varepsilon)$ and satisfies $\varphi\equiv1$ in $(-\varepsilon/2,\varepsilon/2)$, for some $\varepsilon >0$.
\end{remark}

\begin{remark}\label{genenba}
When in equation \eqref{intro-eq1} the matrix $A$ is (possibly) not the identity, an energy balance similar to \eqref{lunedi} is still valid: for every $t \in [0,1]$, there holds
\begin{equation} 
\label{gedb}
\mathcal{E}(t) - \mathcal{E}(0) + \frac{\pi}{4}\int_0^t k^2(\tau) a(\tau) \dot{s}(\tau) \, \de\tau = \int_0^t \langle f(\tau), \dot{u}(\tau) \rangle_{L^2(\Omega)} \, \de\tau\,,
\end{equation}
where $a$ is a function depending only on $A$, $\Gamma$, and $s$, and it is given by
$$
a(t) := |A^{-1/2}(r(t)) \gamma'(s(t))| \cdot|A^{1/2}(r(t)) n(s(t))|\cdot \sqrt{\det A(r(t))}\,.
$$
Here $A^{1/2}$ and $A^{-1/2}$ denote the square root of the symmetric and positive definite matrices $A$ and $A^{-1}$, respectively, and $\gamma'(s(t))$ and $n(s(t))$ are the tangent and normal unit vectors to $\Gamma$ at the point $r(t):=\gamma(s(t))$, respectively. In this case, the energy-dissipation balance \eqref{edb} holds true if and only if the stress intensity factor $k(t)$ satisfies
$$
k(t) = \frac{2}{\sqrt{\pi a(t)}}
$$
during the crack opening, namely when $\dot{s}(t) > 0$. 

In order to derive formula \eqref{gedb}, we use the decomposition result \eqref{decomp1} rewritten as
$$
u(t,x) = u^R(t,x) + k(t)\zeta(t,x) \overline{S}(t, \chi(x))\,,
$$
where $\overline{S}(t,x)$ is the singular part of the solution relative to the transformed curve $\Gamma^{(1)}=\chi(\Gamma)$. Then we proceed as in the previous theorem and proposition: we test the PDE with $\dot{u}(t) \varphi_\epsilon$ (where $\varphi_\epsilon(x) = \frac{\text{dist}(x,\Gamma(\overline t) \setminus \Gamma(0))}{\epsilon} \wedge 1$), and as before, we note that the only delicate term is the one that converges to the integral in the left hand-side of \eqref{gedb}:
$$
\lim_{\epsilon \to 0^+} \int_0^{\overline t} k^2(t)\bigg(\int_{\Omega}\zeta^2(t,x) [A(x) \nabla \overline{S}(t, \chi(x)) \cdot \nabla \varphi_{\epsilon}(x)] \, \dot{\overline{S}}(t,\chi(x)) \, \de x\bigg) \,\de t\,.
$$
By applying the change of variables $\chi^{-1}$, we can rewrite the space integral in the previous expression as follows:
$$
\int_{\Omega} \zeta^2(t,x) [D\chi A D\chi^T](\chi^{-1}(x))\nabla \overline{S}(t,x) \cdot D\chi^{-T}(x) \nabla \varphi_{\epsilon}(\chi^{-1}(x)) \, \dot{\overline{S}}(t,x) |J \chi^{-1}(x)| \, \de x\,.
$$
Finally, we work on the transformed curve $\Gamma^{(1)}$, exactly as in the previous theorem, using the property of the singular part $\overline{S}(t,x)$ together with the following facts: by construction, $[D\chi A  D\chi^T ](\chi^{-1}(x))$ is a continuous function which agrees with the identity on the points of $\Gamma^{(1)}$;  $D\chi^{-T}(x) \nabla \varphi_{\epsilon}(\chi^{-1}(x))$ is a continuous function equal to $\frac{1}{\epsilon}|A^{1/2}(r(t)) n(s(t))|\, n^{(1)}(s^{(1)}(t))$ on the points of $\Gamma^{(1)}$, where $n^{(1)}(s^{(1)}(t))$ denotes the normal unit vector to $\Gamma^{(1)}$ at the point $\gamma^{(1)}(s^{(1)}(t))$; the velocity $\dot{s}^{(1)}$ of the curve $\Gamma^{(1)}$ satisfies $\dot{s}^{(1)}(t) = |A^{-1/2}(r(t))\gamma'(s(t))| \dot{s}(t)$; finally, $|J \chi^{-1}(x)|$ is a continuous function equal to $\sqrt{\det A(r(t))}$ on the points of $\Gamma^{(1)}$. 
\end{remark}

\begin{remark} We underline that Proposition \ref{enbalrect}, Theorem \ref{thmeb}, and Remark \ref{genenba} give an important quantitative information on $k$ and $s$: for every $t\in [0,T]$
$$
\left[\frac2{\sqrt{\pi a(t)}} - k(t)\right]\dot{s}(t)=0\,.
$$
In particular, in the set $\{ t \ :\ \dot{s}(t)>0 \}\subset [0,T]$ the stress intensity factor $k$ coincides with the function $2/\sqrt{\pi a}$. On the other hand, nothing can be said for the times for which $\dot{s}=0$.
 \end{remark}

\subsection*{Acknowledgments.} The first two authors are members of the {\em Gruppo Nazionale per l'Analisi Ma\-te\-ma\-ti\-ca, la Probabilit\`a e le loro Applicazioni} (GNAMPA) of the {\em Istituto Nazionale di Alta Matematica} (INdAM).
The work of M. Caponi has been partially supported by the INdAM -- GNAMPA project 2018 {\em Problemi non lineari alle derivate parziali} (Prot\_U-UFMBAZ-2018-000384).

\end{document}